\documentclass[letterpaper, 10 pt, conference]{ieeeconf}  % Comment this line out
\IEEEoverridecommandlockouts                              % This command is only
\usepackage{bbm}
\usepackage{mathrsfs}
\usepackage{verbatim}
\usepackage{amsmath}
\usepackage{amsfonts}
\usepackage{bm}
\usepackage{graphicx}
\usepackage{float}

\usepackage{amsthm}
\usepackage{xcolor}
\usepackage[colorinlistoftodos]{todonotes}
\usepackage{mathtools}
\usepackage{cite}
\usepackage{colortbl}
\usepackage{subfigure}
\allowdisplaybreaks

\theoremstyle{definition}
\newtheorem{assumption}{Assumption}
\newtheorem{definition}{Definition}
\newtheorem{remark}{Remark}
\newtheorem{problem}{Problem}

\theoremstyle{plain}

\newtheorem{theorem}{Theorem}
\newtheorem{lemma}{Lemma}

\let\inf\relax 
\DeclareMathOperator*\inf{\vphantom{p}inf}

\title{\LARGE \bf Worst-Case Control and Learning Using Partial Observations Over an Infinite Time Horizon}

\author{Aditya Dave, Ioannis Faros, Nishanth Venkatesh, {\itshape{Student Members, IEEE,}} \\ Andreas A. Malikopoulos, {\itshape{Senior Member, IEEE}} 
	\thanks{This research was supported by NSF under Grants CNS-2149520 and CMMI-2219761.} %
	\thanks{The authors are with the Department of Mechanical Engineering, University of Delaware, Newark, DE 19716 USA (email: \texttt{adidave@udel.edu; ifaros@udel.edu; nish@udel.edu; andreas@udel.edu).}} }

\begin{document}

\maketitle
\thispagestyle{empty}

\begin{abstract}
Safety-critical cyber-physical systems require control strategies whose worst-case performance is robust against adversarial disturbances and modeling uncertainties.
In this paper, we present a framework for approximate control and learning in partially observed systems to minimize the worst-case discounted cost over an infinite time horizon. 
We model disturbances to the system as finite-valued uncertain variables with unknown probability distributions. For problems with known system dynamics, we construct a dynamic programming (DP) decomposition to compute the optimal control strategy. 
Our first contribution is to define information states that improve the computational tractability of this DP without loss of optimality.
Then, we describe a simplification for a class of problems where the incurred cost is observable at each time instance. 
Our second contribution is defining an approximate information state that can be constructed or learned directly from observed data for problems with observable costs. We derive bounds on the performance loss of the resulting approximate control strategy and illustrate the effectiveness of our approach in partially observed decision-making problems with a numerical example.
%A key feature of these approximate information states is that they can be learned directly from observed data, without knowledge of system dynamics. 
%We illustrate the effectiveness of our approach to learning approximate information states and controlling partially observed systems using a numerical example.
\end{abstract}

\section{Introduction}
\label{section:Introduction}

Cyber-physical systems, such as connected and automated vehicles \cite{Malikopoulos2020} and power systems \cite{zhu2011robust}, often require decision-making in uncertain environments with partial knowledge of the dynamics \cite{Malikopoulos2022a} over long time horizons. This decision-making challenge is typically modeled with a \textit{stochastic formulation}, where an agent can access a prior probability distribution for all uncertainties and computes a control strategy to minimize the expected value of a discounted total cost across an infinite time horizon \cite{Dave2021nestedaccess}. %This approach has also been utilized in reinforcement learning \cite{subramanian2022approximate} and decentralized systems \cite{Dave2021nestedaccess}.
However, the actual performance of such a strategy degrades when the given prior distribution is different from the actual underlying distribution \cite{mannor2007bias}. To mitigate this drawback, research efforts have proposed alternatives, including \textit{(1) robust stochastic formulations,} where an agent minimizes the worst-case expected cost given a set of feasible probability distributions \cite{iyengar2005robust, wiesemann2013robust};
%iyengar2005robust, 
and \textit{(2) risk-averse formulations,} where an agent minimizes a combination of both the expected cost and the cost variance \cite{mihatsch2002risk, baauerle2017partially}. 
While these formulations improve the performance under a distribution mismatch, many safety-critical applications require further guarantees on the worst-case performance of a strategy against either adversarial attacks or system failure, e.g., cyber-security \cite{rasouli2018scalable} and cyber-physical networks \cite{shoukry2013minimax}. %, and power systems \cite{zhu2011robust}. 
A \textit{non-stochastic formulation} is suitable in such applications, where the agent has no knowledge of the distributions on the uncertainties and uses only their set of feasible values to compute a control strategy that minimizes the \textit{maximum} possible cost \cite{bertsekas1973sufficiently, bernhard2000max,  gagrani2017decentralized, coraluppi1999risk, Dave2021minimax, Dave2023approximate}. This non-stochastic formulation is both maximally robust for a given set of uncertainties and most risk-averse under any feasible prior distribution \cite{james1994risk}.
%, coraluppi1999risk}.

%In either of these formulations, either the risk-sensitivity or robustness need to be maximized in applications that require guarantees on a system's worst-case performance, for example: (1) control of systems under attack from an adversary, like cyber-security systems \cite{rasouli2018scalable}, and (2) control of systems where a single event of failure can be damaging, like water reservoirs \cite{giuliani2021state}. 
%The limiting formulation in both cases yields a non-stochastic worst-case formulation \cite{james1994risk, coraluppi1999risk}, where the agent aims to minimize the worst-case discounted cost without knowledge of distributions \cite{bernhard2000max, bacsar2008h, moon2015minimax, gagrani2017decentralized, Dave2021minimax, Dave2023approximate}. 

In this paper, we analyze a non-stochastic problem over an infinite time horizon with an agent that can access only output data and may not know the underlying state-space model. In such partially observed problems, an optimal strategy can be computed using a memory-based dynamic program (DP) when the time horizon is finite. However, as time tends to infinity, the agent's memory grows to an infinite-dimensional vector. This makes a memory-based approach computationally intractable and necessitates an alternate solution. When the state-space model is known, this challenge is alleviated using the maximum cost-to-come as an \textit{information state} in the DP, for both finite-time \cite{bernhard2003minimax} and infinite-time problems \cite{baras1998robust}. The computational tractability of this DP has been further improved in finite-time problems using \textit{approximate information states} \cite{Dave2022approx, dave2022additive}. Meanwhile, general notions of information states and their approximations have been developed for stochastic problems over an infinite time-horizon without relying on state-space models \cite{subramanian2022approximate}. 
In contrast, to the best of our knowledge, no such general notions exist for infinite-time problems to minimize the worst-case discounted total cost. 

%Meanwhile, for stochastic decision-making problems with an unknown state-space model, a general notion of an information state and its approximation has been defined in \cite{subramanian2022approximate}.

Our main contributions in this paper are: (1) we introduce general information states (Definition \ref{def_information_state}) and a time-invariant DP to compute an optimal strategy in non-stochastic problems over an infinite time horizon (Theorem \ref{thm_g_opt});
(2) we specialize information states for problems with observable costs (Definition \ref{def_info_specialized} and Theorem \ref{thm_specialized_bounds}) and define approximate information states (Definition \ref{def_approximate_info}) to compute a strategy with a bounded performance loss (Theorem \ref{thm_approx}); 
and (3) using a numerical example, we show that approximate information states can be learned directly from output data with incomplete access to system dynamics and, subsequently, compute an approximate strategy using deep Q-learning (Section \ref{section:example}).

The remainder of the paper proceeds as follows. In Section \ref{section:problem}, we present our formulation. In Section \ref{section:DP_and_info}, we define information states and the corresponding DP. In Section \ref{section:perfectly_observed}, we specialize our results to observable costs, define approximate information states, and derive performance bounds. 
In Section \ref{section:example}, we demonstrate our results in a numerical example, and in Section \ref{section:conclusion}, we draw concluding remarks. %and discuss ongoing work

\subsection{Notation and Preliminaries}
\label{subsection:Notation}

In our exposition, we use the mathematical framework of uncertain variables \cite{nair2013nonstochastic} and cost distributions \cite{bernhard2003minimax,dave2022additive}:

\vspace{2pt}

\textbf{1) Cost Measures:}
Consider a sample space $\Omega$ with a sigma algebra $\mathcal{B}(\Omega)$. A cost measure is the non-stochastic analogue of a probability measure. Specifically, it is a function $q: \mathcal{B}(\Omega) \to \{-\infty\} \cup (-\infty,0]$ satisfying the properties: (1) $q(\Omega) = 0$, (2) $q(\emptyset) = - \infty$, and (3) $q(B) = \sup_{\omega \in B} q(\omega)$ for all $B \in \mathcal{B}(\Omega)$, where $\sup_{\omega \in \emptyset} q(\omega) := - \infty$. Furthermore, for two sets $B^1, B^2 \in \mathcal{B}(\Omega)$ with $q(B^2) > - \infty$, the conditional cost measure of $B^1$ given $B^2$ is 
$q(B^1\,| \,B^2) := q(B^1, B^2) - q(B^2),$ where $q(B^1, B^2) = \sup_{\omega \in B^1 \cap B^2} q(\omega).$ 

\vspace{2pt}

\textbf{2) Uncertain Variables:}
For a set $\mathscr{X}$, an uncertain variable is a mapping $\mathsf{X}: \Omega \to \mathscr{X}$ and is compactly denoted by $\mathsf{X} \in \mathscr{X}$. This is the non-stochastic equivalent of a random variable. For any $\omega \in \Omega$, its realization is $\mathsf{X}(\omega) = \mathsf{x} \in \mathscr{X}$. 
Its \textit{marginal range} is the set of feasible realizations $[[\mathsf{X}]] \hspace{-1pt} := \hspace{-1pt} \{\mathsf{X}(\omega) \,| \, \omega \in \Omega\} \subseteq \mathscr{X}$.
The \textit{cost distribution} is
$q(\mathsf{x}) := \sup_{\omega \in \{\Omega|\mathsf{X}(\omega) = \mathsf{x}\}} q(\omega)$ for all $\mathsf{x} \in [[\mathsf{X}]]$.
The \textit{joint range} of two uncertain variables $\mathsf{X} \in \mathscr{X}$ and $\mathsf{Y} \in \mathscr{Y}$ is the set of feasible simultaneous realizations $[[\mathsf{X},\mathsf{Y}]] \hspace{-1pt} := \hspace{-1pt} \big\{ \big(\mathsf{X}(\omega), \mathsf{Y}(\omega) \big) \,| \, \omega \in \Omega \big\} \subseteq \mathscr{X} \times \mathscr{Y}$. The two uncertain variables are \textit{independent} if $[[\mathsf{X},\mathsf{Y}]] = [[\mathsf{X}]] \times [[\mathsf{Y}]]$. The \textit{conditional range} of $\mathsf{X}$ given a realization $\mathsf{y}$ of $\mathsf{Y}$ is the set $[[\mathsf{X}|\mathsf{y}]] \hspace{-1pt} := \hspace{-1pt} \big\{ \mathsf{X}(\omega) \; | \; \mathsf{Y}(\omega) = \mathsf{y}, \; \omega \in \Omega \big\}$. %, and in general, $[[\mathsf{X}|\mathsf{Y}]] \hspace{-1pt} := \hspace{-1pt} \big\{ [[\mathsf{X}|\mathsf{y}]]\;| \; \mathsf{y} \in [[\mathsf{Y}]] \big\}$.
The cost distribution of any $\mathsf{x} \in [[\mathsf{X}]]$ given a realization $\mathsf{y} \in [[\mathsf{Y}]]$ with $q(\mathsf{y}) > - \infty$ is $q(\mathsf{x}\,| \,\mathsf{y}) = q(\mathsf{x},\mathsf{y}) - q(\mathsf{y})$, where $q(x,y) = \sup_{\omega \in \{\Omega|X(\omega)=x, Y(\omega) = y\}} q(\omega)$.

\vspace{2pt}

\textbf{3) Hausdorff Distance:}
Consider two bounded, non-empty subsets $\mathscr{X}, \mathscr{Y}$ of a metric space $(\mathscr{S}, \eta)$, where $\eta: \mathscr{X} \times \mathscr{Y} \to \mathbb{R}_{\geq0}$ is the metric. The Hausdorff distance between $\mathscr{X}$ and $\mathscr{Y}$ is the pseudo-metric \cite[Chapter 1.12]{barnsley2006superfractals}:
\begin{align} \label{H_met_def}
    \hspace{-4pt} \mathcal{H}(\mathscr{X}, \mathscr{Y}) \hspace{-2pt} := \hspace{-2pt} \max \hspace{-1pt} \big\{ \hspace{-1pt} \sup_{\mathsf{x} \in \mathscr{X}} \hspace{-1pt} \inf_{\mathsf{y} \in \mathscr{Y}}  \eta(\mathsf{x}, \mathsf{y}), \sup_{\mathsf{y} \in \mathscr{Y}} \hspace{-1pt} \inf_{\mathsf{x} \in \mathscr{X}}  \eta(\mathsf{x}, \mathsf{y}) \hspace{-2pt} \big\}. \hspace{-4pt}
\end{align} 
Furthermore, if $f: \mathscr{S} \to \mathbb{R}$ is a Lipschitz continuous function with a constant $L_f \in \mathbb{R}_{\geq0}$, then \cite[Lemma 5]{Dave2023approximate}:
\begin{align} \label{H_met_property}
    \Big|\sup_{x \in \mathscr{X}}f(x) - \sup_{y \in \mathscr{Y}}f(y)\Big| \leq L_f \cdot \mathcal{H}(\mathscr{X}, \mathscr{Y}).
\end{align}

\section{Problem Formulation} \label{section:problem}

We consider the control of an uncertain system which evolves in discrete time steps. At each time $t \in \mathbb{N} = \{0,1,2,\dots\}$ an agent collects an observation on the system as the uncertain variable $Y_t \in \mathcal{Y}$ and generates a control action denoted by the uncertain variable $U_t \in \mathcal{U}$. After generating the action at each $t$, the agent incurs a cost denoted by the uncertain variable $C_t \in \mathcal{C} \subset \mathbb{R}_{\geq0}$. The set $\mathcal{C}$ is bounded by $\min\{\mathcal{C}\} = c^{\min}$ and $\max\{\mathcal{C}\} = c^{\max}$. We formulate our problem for a general case where the agent may not have knowledge of a state-space model for the system. Thus, we use an \textit{input-output model} to describe the evolution of the system, as follows. At each ${{t \in \mathbb{N}}}$, the system receives two inputs: the action $U_t$, and an uncontrolled disturbance $W_t \in \mathcal{W}$. The disturbances $\{W_t\,| \,{{t \in \mathbb{N}}}\}$ constitute a sequence of independent uncertain variables.
After receiving the inputs at each time $t \in \mathbb{N}$, the system generates two outputs: (1) the observation $Y_{t+1} = h_{t+1}(W_{0:t}, U_{0:t})$, where $h_{t+1}: \mathcal{W}^t \times \mathcal{U}^t \to \mathcal{Y}$ is the observation function; and (2) the cost $C_t = d_t(W_{0:t}, U_{0:t})$, where $d_t: \mathcal{W}^t \times \mathcal{U}^t \to \mathcal{C}$ is the cost function. The initial observation is generated as $Y_0 = h_0(W_0)$.

The agent perfectly recalls all observations and control actions and at each ${{t \in \mathbb{N}}}$, the agent's memory is the uncertain variable $M_t := (Y_{0:t}, U_{0:t-1})$ taking values in $\mathcal{M}_t := \mathcal{Y}^t \times$ $ \mathcal{U}^{t-1}$. The agent uses a control law $g_t: \mathcal{M}_t \to \mathcal{U}$ to generate the action $U_t = g_t(M_t)$ as a function of the memory. The control strategy is the collection of control laws $\boldsymbol{g} := (g_0, g_1,\dots)$ with a feasible set $\mathcal{G}$. The performance of a strategy $\boldsymbol{g} \in \mathcal{G}$ is given by the \textit{worst-case discounted cost},
    \begin{align} \label{eq_instantaneous_criterion}
    \mathcal{J}(\boldsymbol{g}) := \lim_{T \to \infty} \sup_{c_{0:T} \in [[C_{0:T}]]^{\boldsymbol{g}}} \sum_{t=0}^T \gamma^t {\cdot} c_t,
    \end{align}
where $\gamma \in (0,1)$ is a discount parameter, the marginal range $[[C_{0:T}]]^{\boldsymbol{g}}$ is the set of all feasible costs consistent with the strategy $\boldsymbol{g}$ and with the set of feasible disturbances $\mathcal{W}$. The limit in \eqref{eq_instantaneous_criterion} is well defined because $C_t \leq c^{\max}$ for all $t$. Next, we define the control problem with known dynamics.

\begin{problem} \label{problem_1}
The optimization problem is to derive the infimum value $\inf_{\boldsymbol{g} \in \mathcal{G}} \mathcal{J}(\boldsymbol{g}),$
given the feasible sets $\{\mathcal{U}, \mathcal{W}, \mathcal{Y}, \mathcal{C}\}$ and the functions $\{h_t, d_t \,| \, {{t \in \mathbb{N}}} \}$.
\end{problem}

If the minimum value is achieved in Problem \ref{problem_1}, the minimizing argument $\boldsymbol{g}^* = \arg \min_{\boldsymbol{g} \in \mathcal{G}} \mathcal{J}(\boldsymbol{g})$ is called an optimal control strategy. Our aim is to tractably compute the optimal value and an optimal strategy, if one exists. We impose the following assumption in our analysis.

\begin{assumption} \label{assumption_1}
We consider that the sets $\{\mathcal{U},\mathcal{W},\mathcal{Y}\}$ are each bounded subsets of a metric space $(\mathscr{S}, \eta)$ and $\mathcal{C}$ is a bounded subset of $\mathbb{R}_{\geq0}$.
\end{assumption}

Assumption \ref{assumption_1} ensures that all uncertain variables take values in bounded sets and that we can use the Hausdorff pseudo-metric \eqref{H_met_def} as a distance measure between them. 

\begin{remark}
We first derive results for Problem \ref{problem_1} with known dynamics. However, our main results in Section \ref{section:perfectly_observed} are also suitable for reinforcement learning problems with unknown dynamics. We illustrate this with an example in Section \ref{section:example}.
\end{remark}

\section{Dynamic Program and Information States} \label{section:DP_and_info}

In this section, we first present value functions to evaluate the performance of any strategy $\boldsymbol{g} \in \mathcal{G}$. Next, we present a memory-based DP decomposition of Problem \ref{problem_1} that approximately computes the value functions with arbitrary precision. However, because the memory grows in size with time, this DP suffers from exponentially increasing computation with an increase in precision. To alleviate this computational challenge, we present the notion of information states in Subsection \ref{subsection:basic_info_states}.
%Note that a memory-based DP cannot compute the optimal value exactly because the memory grows to an infinite-dimensional vector as $t \to \infty$. Then, we define time-invariant approximate information states for this problem and show how they can be used to to simplify the DP decomposition. 
To construct value functions, we first define the \textit{accrued cost} at each ${{t \in \mathbb{N}}}$ as the sum of past incurred costs $A_t := \sum_{\ell=0}^{t-1} \gamma^\ell {\cdot} C_{\ell},$
which satisfies $A_{t+1} = A_t + \gamma^t{\cdot} C_t$ with $A_0 := 0$. 
This is well defined in the limit $t \to \infty$ because $\lim_{t \to \infty} A_t \leq \lim_{t \to \infty} \sum_{\ell=0}^{t-1} \gamma^{\ell} {\cdot} c^{\max} = \frac{c^{\max}}{1-\gamma} =: a^{\max}$. Thus, $A_t \in [0, a^{\max}]$ for all ${{t \in \mathbb{N}}}$. Similarly, the \textit{cost-to-go} at any ${{t \in \mathbb{N}}}$ is the sum of future all costs still to be incurred $C^{\infty}_t := \sum_{\ell = t}^{\infty} \gamma^{\ell-t} {\cdot} C_{\ell}.$
Note that $C^{\infty}_t \in [0, a^{\max}]$ for all $t$ and that %starting with $C^{\infty}_t = 0$, the cost-to-go follows the dynamics
$C^{\infty}_{t} = C_t +  \gamma {\cdot} C^{\infty}_{t+1}$. Then, for all ${{t \in \mathbb{N}}}$, we can define a value function for any $\boldsymbol{g} \in \mathcal{G}$ as 
    \begin{gather} \label{value_function_t_g}
        V_t^{\boldsymbol{g}}(m_t) := \sup_{a_t, c^{\infty}_{t} \in [[A_t, C^{\infty}_{t}|m_t]]^{\boldsymbol{g}}} \left(a_t + \gamma^t {\cdot} c^{\infty}_t \right),
    \end{gather}
where $[[A_t, C^{\infty}_{t}|m_t]]^{\boldsymbol{g}}$ is the conditional range induced by the choice of strategy $\boldsymbol{g}$. 
From the definition of the value functions, at $t=0$ it holds that $\sup_{y_0 \in \mathcal{Y}}V_0^{\boldsymbol{g}}(y_0) = \mathcal{J}({\boldsymbol{g}})$, where $m_0 = y_0$. 
Thus, the value function $V_0^{\boldsymbol{g}}(y_0)$ evaluates the performance of any strategy $\boldsymbol{g}$ for an initial observation $y_0$. Similarly, the optimal value function at each ${{t \in \mathbb{N}}}$ is
\begin{gather} \label{value_function_t_opt}
    V_t(m_t) := \inf_{\boldsymbol{g} \in \mathcal{G}} V_t^{\boldsymbol{g}}(m_t),
\end{gather}
and the optimal value %in Problem \ref{problem_1} 
is $\inf_{\boldsymbol{g} \in \mathcal{G}} \mathcal{J}(\boldsymbol{g}) = \sup_{y_0 \in \mathcal{Y}} V_0(y_0)$. 

Given the value functions in \eqref{value_function_t_g} and \eqref{value_function_t_opt}, we can evaluate the performance of a strategy and compare it with the optimal performance. However, there is no natural DP decomposition to compute these value functions in an infinite-horizon system with no terminal time. 
Thus, we construct a memory-based DP that assumes a finite horizon $T \in \mathbb{N}$ and use it to recursively compute approximations of the value functions. For any $\boldsymbol{g} \in \mathcal{G}$, we define finite-horizon evaluation functions for all $m_t \in \mathcal{M}_t$ and each $t=0,\dots,T-1$ as
    \begin{gather}
        J_{t}^{\boldsymbol{g}}(m_t; T) := \sup_{m_{t+1} \in [[M_{t+1}|m_t]]^{\boldsymbol{g}}} J_{t+1}^{\boldsymbol{g}}(m_{t+1}; T),
    \end{gather}
where $J_{T}^{\boldsymbol{g}}(m_T;T) \hspace{-1pt} := \hspace{-1pt} \sup_{a_T, c_{T} \in [[A_T, C_{T}|m_T]]^{\boldsymbol{g}}} \hspace{-1pt} (a_T + \gamma^T {\cdot} c_T)$. Similarly, we define approximately optimal finite-horizon functions for all $m_t \in \mathcal{M}_t$ and each $t=0,\dots, T-1$ as
    \begin{gather} \label{eq_DP_memory_opt}
        J_{t}(m_t; T) := \inf_{u_t \in \mathcal{U}}\sup_{m_{t+1} \in [[M_{t+1}|m_t, u_t]]} J_{t+1}(m_{t+1}; T),
    \end{gather}
where, $J_{T}(m_T;T) := \inf_{u_T \in \mathcal{U}} \sup_{a_T, c_{T} \in [[A_T, C_{T}|m_T, u_T]]}$ $(a_T + \gamma^T {\cdot} c_T)$. %If the minimum is achieved in the RHS of \eqref{eq_DP_memory_opt}, the minimizing argument gives an approximately optimal control law $g_t:\mathcal{M}_t \to \mathcal{U}$ for all $t=0,\dots,T$. 
Note that the finite-horizon functions $J_t^{\boldsymbol{g}}(m_t; T)$ and $J_t(m_t; T)$ at any $t=0,\dots,T$ are parameterized by the choice of horizon $T \in \mathbb{N}$. Next, we bound the approximation error between the value functions and their finite-horizon counterparts.

% minimizing argument $u_t^*$ in the RHS of \eqref{eq_DP_memory_opt} exists for all $t=0,\dots,T$, then the corresponding approximate control law is $g_t(m_t) = u_t^*$. %Here, note that the the presence of the control action in the conditioning ensures that the conditional ranges are independent of the choice of strategy $\boldsymbol{g}$. 
%Next, we show that by increasing the horizon $T \in \mathbb{N}$, we can estimate the optimal value functions with arbitrary accuracy using the finite-horizon functions.

\begin{lemma} \label{finite_DP_memory}
For any finite horizon $T \in \mathbb{N}$ and for all $m_t \in \mathcal{M}_t$ and each $t=0,\dots,T$, 
    \begin{align} 
        \hspace{-8pt} \textbf{a) }\dfrac{\gamma^{T+1} \hspace{-1pt} {\cdot} c^{\min}}{1-\gamma} \hspace{-2pt} + \hspace{-2pt} J_{t}^{\boldsymbol{g}}(m_t; \hspace{-1pt} T) &\leq \hspace{-1pt} V_t^{\boldsymbol{g}}(m_t) \nonumber \\
        &\leq \hspace{-1pt} J_{t}^{\boldsymbol{g}}(m_t; \hspace{-1pt} T) \hspace{-2pt} + \hspace{-2pt}\dfrac{\gamma^{T+1} \hspace{-1pt} {\cdot} c^{\max}}{1-\gamma}, \label{eq_finite_DP_memory_g} \\
        \hspace{-10pt} \textbf{b) } \dfrac{\gamma^{T+1} \hspace{-1pt}  {\cdot} c^{\min}}{1-\gamma} \hspace{-1pt} + \hspace{-1pt} J_{t}(m_t ; \hspace{-1pt} T) \hspace{-1pt}  &\leq \hspace{-1pt}  V_t(m_t) \nonumber \\
        &\leq \hspace{-1pt}  J_{t}(m_t; \hspace{-1pt} T) \hspace{-1pt}  + \hspace{-1pt} \dfrac{\gamma^{T+1} \hspace{-1pt}  {\cdot} c^{\max}}{1-\gamma}. \hspace{-4pt} \label{eq_finite_DP_memory_opt}
    \end{align}
\end{lemma}

\begin{proof}
\textit{a)} We prove each inequality in \eqref{eq_finite_DP_memory_g} using backward induction. For the upper bound at time $T$, we use the dynamics of the accrued cost and cost-to-go to write that $V^{\boldsymbol{g}}_T(m_T) = \sup_{a_T, c_T, c_{T+1}^{\infty} \in [[A_T, C_T, C_{T+1}^{\infty}|m_T]]^{\boldsymbol{g}}} (a_T + \gamma^T {\cdot} c_T + \gamma^{T+1} {\cdot} c_{T+1}^\infty)$ $ \leq \sup_{a_T, c_T \in [[A_T, C_T|m_T]]^{\boldsymbol{g}}} \left(a_T + \gamma^T {\cdot} c_T\right) + \gamma^{T+1} {\cdot} a^{\max} \leq J_{t}^{\boldsymbol{g}}(m_t;T) + \frac{\gamma^{T+1}{\cdot} c^{\max}}{1-\gamma}$. The lower bound at time $T$ follows from $\frac{c^{\min}}{1-\gamma} \leq c_{T+1}^{\infty}$ using the same sequence of arguments as before. This forms the basis of our induction.
Next, consider the hypothesis that \eqref{eq_finite_DP_memory_g} holds at time $t+1$. For the upper bound at time $t$, by definition $V_t^{\boldsymbol{g}}(m_t) 
= \sup_{a_t, c_t, c^{\infty}_{t+1} \in [[A_t, C_t, C_{t+1}^{\infty}|m_t]]^{\boldsymbol{g}}} \left(a_t + \gamma^t {\cdot} c_t + \gamma^{t+1} {\cdot} c^{\infty}_{t+1} \right) 
= \sup_{a_{t+1}, c^{\infty}_{t+1} \in [[A_{t+1}, C^{\infty}_{t+1}|m_t]]^{\boldsymbol{g}}}(a_{t+1} + \gamma^{t+1} {\cdot} c^{\infty}_{t+1}) 
= \sup_{m_{t+1} \in [[M_{t+1}|m_t]]^{\boldsymbol{g}}} \sup_{a_{t+1}, c^{\infty}_{t+1} \in [[A_{t+1}, C^{\infty}_{t+1}|m_{t+1}]]^{\boldsymbol{g}}}(a_{t+1} $ $+ \gamma^{t+1} {\cdot} c^{\infty}_{t+1}) 
= \sup_{m_{t+1} \in [[M_{t+1}|m_t]]^{\boldsymbol{g}}}V^{\boldsymbol{g}}_{t+1}(m_{t+1}) 
\leq \sup_{m_{t+1} \in [[M_{t+1}|m_t]]^{\boldsymbol{g}}} J^{\boldsymbol{g}}_{t+1}(m_{t+1}; T) + \frac{\gamma^{T+1}{\cdot} c^{\max}}{1-\gamma}
= J^{\boldsymbol{g}}_{t}(m_{t}; T) + \frac{\gamma^{T+1}{\cdot} c^{\max}}{1-\gamma}$, where, in the fourth equality, we use \eqref{value_function_t_g} for $V_{t+1}^{\boldsymbol{g}}(m_{t+1})$; and in the inequality, we use the hypothesis. The lower bound follows from the same sequence of arguments. Thus, \eqref{eq_finite_DP_memory_g} holds using induction.

\textit{b)} We can prove the lower bound in \eqref{eq_finite_DP_memory_opt} by taking the infimum on both sides of the lower bound in \eqref{eq_finite_DP_memory_g}. To prove the upper bound in \eqref{eq_finite_DP_memory_opt}, we first note that $J_t(m_t; T) = \inf_{\boldsymbol{g} \in \mathcal{G}} J_t^{\boldsymbol{g}}(m_t; T)$ for all $t=0,\dots,T$ using standard DP arguments for terminal-cost problems \cite{bernhard2003minimax}. Then, at time $T$, by definition $V_T(m_T) = \inf_{\boldsymbol{g} \in \mathcal{G}}V^{\boldsymbol{g}}_T(m_T) \leq \inf_{\boldsymbol{g} \in \mathcal{G}} J^{\boldsymbol{g}}_T(m_T; T) + \frac{\gamma^{T+1}{\cdot} c^{\max}}{1-\gamma} = J_T(m_T; T) + \frac{\gamma^{T+1}{\cdot} c^{\max}}{1-\gamma}$. Using this as the basis, the result follows for all $t=0,\dots,T$ using the same induction arguments as in \eqref{eq_finite_DP_memory_g}. 
\end{proof}

%\begin{remark}
Lemma \ref{finite_DP_memory} establishes that the approximation error between finite-horizon functions and corresponding value functions decreases as the horizon $T \in \mathbb{N}$ increases. A direct consequence of \eqref{eq_finite_DP_memory_opt} is that $\lim_{T \to \infty} J_0(y_0; T) = V_0(y_0)$ for all $y_0 \in \mathcal{Y}$. Note, however, that the domain of $J_{T}(m_T; T)$ is $\mathcal{M}_T = \mathcal{Y}^T \times \mathcal{U}^{T-1}$ which grows with $T$, and in the limit $T \to \infty$, the set $\mathcal{M}_T$ is infinite-dimensional. Thus, it is computationally intractable to achieve close approximations of the optimal value using \eqref{eq_DP_memory_opt}. We address this issue in the next subsection using \textit{information states}, which take values in time-invariant spaces.

% involves computing value functions for each $m_t \in \mathcal{M}_t$ for all $t=0,\dots,T$. As $T \to \infty$, note that the memory becomes an infinite-dimensional vector rendering this a computationally intractable problem for large horizons. In the next subsection, we define \textit{information states} which can address this concern.
%\end{remark}

%    \begin{gather}
%        Q_t^{\boldsymbol{g}}(m_t, u_t) := \sup_{a_t, c_{t:\infty} \in [[A_t, C_{t:\infty}|m_t, u_t]]^g} \left(a_t + \sum_{\ell=t}^\infty \gamma^\ell {\cdot} c_\ell \right).
%    \end{gather}
%    Then $\sup_{y_0}V_0^g(y_0) = \mathcal{J}(g)$. Then, we can prove that, at each ${{t \in \mathbb{N}}}$:
%    \begin{gather*}
%        Q_t^{\boldsymbol{g}}(m_t, u_t) = \sup_{m_{t+1} \in [[M_{t+1}|m_t,u_t]]} V_{t+1}^g(m_{t+1}).
%    \end{gather*}
%Similarly, consider the optimal value functions $V_t(m_t) := \inf_{g} V_t^g(m_t)$ and 
%    \begin{gather*}
%        Q_t(m_t, u_t) := \sup_{m_{t+1} \in [[M_{t+1}|m_t,u_t]]} V_{t+1}(m_{t+1}).
%    \end{gather*}

\subsection{Information States} \label{subsection:basic_info_states}

In this subsection, we present the notion of information states which take values in time-invariant spaces. Then, we use them to construct a time-invariant DP decomposition which converges to the optimal value of Problem \ref{problem_1}.
%construct a fixed-point equation that computes the optimal value in Problem \ref{problem_1}.
%and use them to construct a time-invariant fixed-point operator which can compute the optimal value and optimal control strategy. 
To begin, recall from Section \ref{subsection:Notation} that a \textit{cost distribution} is the non-stochastic equivalent of a probability distribution for uncertain variables. We use two specific cost distributions in our exposition, defined as follows.

\begin{definition} \label{def_ind}
Let $\mathsf{X} \in \mathscr{X}$ and $\mathsf{Y} \in \mathscr{Y}$ be two uncertain variables. The \textit{indicator function} for $\mathbb{I}: \mathscr{X} \to \{\infty, 0\}$ for $\mathsf{x} \in \mathscr{X}$ is given by
\begin{gather}
    \mathbb{I}(\mathsf{x}) :=
    \begin{aligned}
    \begin{cases}
        0, &\text{ if } \mathsf{x} \in [[\mathsf{X}]], \\
        - \infty, &\text{ if } \mathsf{x} \not\in [[\mathsf{X}]],
    \end{cases}
    \end{aligned}
\end{gather}
and the conditional indicator function $\mathbb{I}: \mathcal{X} \times \mathcal{Y} \to  \{\infty, 0\}$ for any $\mathsf{x} \in \mathscr{X}$ given a realization $\mathsf{y} \in [[\mathsf{Y}]]$ is given by
\begin{gather} \label{indicator_def}
    \mathbb{I}(\mathsf{x}\,| \,\mathsf{y}) :=
    \begin{aligned}
    \begin{cases}
        0, &\text{ if } \mathsf{x} \in [[\mathsf{X}\,| \,\mathsf{y}]], \\
        - \infty, &\text{ if } \mathsf{x} \not\in [[\mathsf{X}\,| \,\mathsf{y}]].
    \end{cases}
    \end{aligned}
\end{gather}
\end{definition}

The indicator function verifies whether the input takes values within the conditional range of an uncertain variable and it satisfies the properties of a cost distribution from Subsection \ref{subsection:Notation}. Next, we use it to define the accrued distribution.

\begin{definition} \label{def_accrued_cost}
Let $\mathsf{X} \in \mathscr{X}$ and $\mathsf{Y} \in \mathscr{Y}$ be two uncertain variables and let $A_t \in \mathcal{A}$ be the accrued cost at any $t\in \mathbb{N}$. An \textit{accrued distribution} for any $\mathsf{x} \in \mathscr{X}$ at any $t \in \mathbb{N}$ is a function $r_t: \mathscr{X} \to \{-\infty\} \cup [-a^{\max}, 0]$, given by
\begin{gather}
    r_t(\mathsf{x}) := \sup_{a_t \in \mathcal{A}}\big(a_t + \mathbb{I}(\mathsf{x}, a_t)\big) - \sup_{a_t \in \mathcal{A}} \big(a_t + \mathbb{I}(a_t) \big), 
\end{gather}
and the conditional accrued distribution for $\mathsf{x} \in \mathscr{X}$ given a realization $\mathsf{y} \in [[\mathsf{Y}]]$ is a function $r_t: \mathscr{X} \times \mathscr{Y} \to \{-\infty\} \cup [-a^{\max}, 0]$, given by
\begin{align}
    \hspace{-5pt} r_t(\mathsf{x} \,| \, \mathsf{y}) \hspace{-2pt} := \hspace{-2pt} \sup_{a_t \in \mathcal{A}} \hspace{-2pt} \big(a_t + \mathbb{I}(\mathsf{x}, a_t\,|\,\mathsf{y}) \big) \hspace{-2pt}
    - \hspace{-2pt} \sup_{a_t \in \mathcal{A}} \hspace{-2pt} \big(a_t + \mathbb{I}(a_t\,|\,\mathsf{y}) \big). \hspace{-4pt} \label{def_r}
\end{align}
\end{definition}

The accrued distribution returns $-\infty$ when the input is not within the range of an uncertain variable and it returns an output from $[-a^{\max}, 0]$ otherwise. It also satisfies all properties of a cost distribution as defined in Subsection \ref{subsection:Notation}. Note that, at each ${{t \in \mathbb{N}}}$, given the realizations $m_t \in \mathcal{M}_t$, $u_t \in \mathcal{U}$ and the dynamics, we can compute the indicator $\mathbb{I}(c_t, m_{t+1}\,|\,m_t, u_t)$, and the accrued distribution $r_t(c_t, m_{t+1}\,|\,m_t, u_t)$ for all $c_t \in \mathcal{C}$ and $m_{t+1} \in \mathcal{M}_{t+1}$. We use these accrued distributions to define information states.

\begin{definition} \label{def_information_state}
An \textit{information state} at any $t \in \mathbb{N}$ is an uncertain variable $S_t= \sigma_t(M_t)$ taking values in a bounded, time-invariant subset $\mathcal{S}$ of a metric space $(\mathscr{S}, \eta)$. %, where $\sigma_t: \mathcal{M}_t \to \mathcal{S}$. 
Furthermore, there exists a time-invariant cost distribution $\rho: \mathcal{C} \times \mathcal{S} \times \mathcal{S} \times \mathcal{U} \to \{-\infty\} \cup [-a^{\max}, 0]$ such that, for all $t \in \mathbb{N}$, for all $m_t \in \mathcal{M}_t$, $u_t \in \mathcal{U}$, $c_t \in \mathcal{C}$ and ${{{{s}}}}_{t+1} \in \mathcal{S}$, it satisfies
    \begin{align}
        r_t(c_t, {{{{s}}}}_{t+1}\,|\,m_t, u_t) &= \rho(c_t, {{{{s}}}}_{t+1}\,|\,\sigma_t(m_t), u_t).
    \end{align}
    %where $r_t$ is the accrued distribution at time $t$ and
    %\begin{multline}
    %    \rho(c_t, {{{{s}}}}_{t+1}|\sigma_t(m_t), u_t) = \sup_{a_t \in \mathcal{A}}(a_t + \mathbb{I}(c_t, {{{{s}}}}_{t+1}, a_t|{{{{s}}}}_t, u_t) ) \\
    %    - \sup_{a_t \in \mathcal{A}}(a_t + \mathbb{I}(a_t|{{{{s}}}}_t, u_t)),
    %\end{multline}
    %and $\mathbb{I}(a_t|{{{{s}}}}_t, u_t) = \mathbb{I}(a_t|{{{{s}}}}_t)$ is time-invariant for all ${{t \in \mathbb{N}}}$ and ${{{{s}}}}_t \in \mathcal{S}$ and $u_t \in \mathcal{U}$.
\end{definition}

%We call such an information state time-invariant because its feasible set, its effect on the incurred cost and its accrued distribution on its evolution are unchanging with time. Note that for all ${{t \in \mathbb{N}}}$, an information state takes values in a time-invariant space $\mathcal{S}$ and induces a time-invariant accrued distribution $\rho$.
Next, we use the information state to construct a time-invariant operator $\mathcal{T}$ that yields a fixed-point equation to recursively compute the optimal value in Problem \ref{problem_1}. 
We first define the deterministic \textit{cumulative discount} at any ${{t \in \mathbb{N}}}$ as $z_t := \gamma^t \in (0,1]$, where $z_0 = 1$ and $z_{t+1} = \gamma {\cdot} z_t$. 
Then, for any uniformly bounded function $\Lambda:\mathcal{S} \times (0,1] \to \mathbb{R}$ we define $\mathcal{T}:[\mathcal{S} \times (0,1] \to \mathbb{R}] \to [\mathcal{S} \times (0,1] \to \mathbb{R}]$, such that
    \begin{multline} \label{eq_general_value_operator}
        \hspace{-10pt} [\mathcal{T} \Lambda]({{{{s}}}}, z) \\
        \hspace{-10pt} := \underset{u \in \mathcal{U}}{\inf} \; \underset{c \in \mathcal{C}, \, s' \in \mathcal{S}}{\sup} \big(c + \gamma {\cdot} \Lambda(s', \gamma {\cdot} z)
            + {\rho(c, s' \,|\, s, u)}{\cdot}{z^{-1}}\big), \hspace{-5pt}
    \end{multline}
for all $s \in \mathcal{S}$ and $z \in (0,1]$. Note that we use time-invariant notation for all variables in \eqref{eq_general_value_operator} because the sets and functions in the RHS are time-invariant.

\begin{remark}
    The RHS of \eqref{eq_general_value_operator} is finite in the limit $z \to 0$ because $c$ has a finite upper bound, the function $\Lambda$ is uniformly bounded, and $\sup_{c \in \mathcal{C}, s' \in \mathcal{S}} \rho(c, s'\,|\,s,u)= 0$ for all $s \in \mathcal{S}$ and $u \in \mathcal{U}$ from the definition of cost distributions.
\end{remark}

Due to discounting, $\mathcal{T}$ is a contraction mapping (see proof in Appendix A) and therefore, using the Banach fixed point theorem, the equation $\Lambda = \mathcal{T} \Lambda$ admits a unique solution $\Lambda^{{{\infty}}} = \mathcal{T}\Lambda^{\infty}$. Starting at $\Lambda^{{{0}}}(s, z) := 0$, the fixed-point iteration around $\mathcal{T}$ generates a sequence of functions
    \begin{align} \label{n_iterated_function}
        \Lambda^{{{n+1}}}(s, z) = [\mathcal{T} \Lambda^{{{n}}}](s, z)
        = [\mathcal{T}^n \Lambda^{{{0}}}](s,z),
    \end{align}
for all $n=1,2,\dots$, such that $\lim_{n \to \infty} \mathcal{T}^n V^{{{0}}} = \Lambda^{{{\infty}}}$.
The fixed-point iteration in \eqref{n_iterated_function} forms a time-invariant DP.
Next, we establish that $\Lambda^n(\sigma_t(m_t), z_t)$, for any ${{n \in \mathbb{N}}}$, can be used to estimate the value function $V_t(m_t)$ of Problem \ref{problem_1} at any ${{t \in \mathbb{N}}}$, with estimation error that decreases in $n$.
%error bounds for any ${{t \in \mathbb{N}}}$, when $\Lambda^n(\sigma_t(m_t), z_t)$, $z_t = \gamma^t$, for any ${{n \in \mathbb{N}}}$, is used to estimate the value function $V_t(m_t)$ of Problem \ref{problem_1} from \eqref{value_function_t_opt}.

%\begin{remark}
%    Note that in the limit $z \to 0$, the RHS of \eqref{eq_general_value_operator} has the limit $\inf_{u \in \mathcal{U}}\sup \big\{c + \gamma {\cdot} \Lambda(s', 0)~|~\rho(c, s'|s,u)=0\big\}$. This limit is well defined because from the definition of cost distributions $\sup_{s' \in \mathcal{S}, c \in \mathcal{C}} \rho(c, s'|s,u)= 0$. 
%\end{remark}

%Thus, the sequence of iterations $\mathcal{T}^n\Lambda^{{{0}}}({{{{s}}}}, z)$ for $n = 1,2,\dots$ converges to a fixed point given by
%    \begin{gather}
%       V^{{{\infty}}}({{{{s}}}}, z) := \min_{u \in \mathcal{U}} \sup_{a_t, c_{0:T} \in [[A_t, C_{0:T}|{{{{s}}}}, u]]} \left(a_t + z {\cdot} \sum_{\ell=0}^T \gamma^\ell {\cdot} c_{\ell} \right),
%    \end{gather}
%    for all ${{{{s}}}}_t \in \mathcal{S}$ and ${{t \in \mathbb{N}}}$.

\begin{theorem} \label{thm_g_opt}
    Consider the function $\Lambda^n$, for any ${{n \in \mathbb{N}}}$, generated using \eqref{n_iterated_function}. Then, for all ${{t \in \mathbb{N}}}$, it holds that
    \begin{multline} \label{eq_gen_info_bounds}
        \hspace{-10pt} \frac{\gamma^{n+t}{\cdot} c^{\min}}{1-\gamma} + \gamma^t {\cdot} \Lambda^{{{n}}}(\sigma_t(m_t), \gamma^t) + \hspace{-2pt} \sup_{a_t \in [[A_t|m_t]]} \hspace{-4pt} a_t \; \leq \; V_t(m_t) \\ 
        \hspace{-10pt}\leq \sup_{a_t \in [[A_t|m_t]]}a_t + \gamma^t {\cdot} \Lambda^{{{n}}}(\sigma_t(m_t), \gamma^t)  + \frac{\gamma^{n+t} {\cdot} c^{\max}}{1-\gamma}.
    \end{multline}
\end{theorem}

\begin{proof}
We show \eqref{eq_gen_info_bounds} using \eqref{eq_finite_DP_memory_opt} from Lemma \ref{finite_DP_memory}. For this purpose, we first show that for any finite horizon $T \in \mathbb{N}$, the following property holds for each $t=0,\dots,T$:
    \begin{gather} \label{interim_gen_info_bounds_1}
        J_{t}(m_t; T) = \gamma^t {\cdot} \Lambda^{T-t+1}\big(\sigma_t(m_t), \gamma^t\big) + \hspace{-4pt} \sup_{a_t \in [[A_t|m_t]]}a_t,
    \end{gather}
where $\Lambda^{T-t+1}$ is the $(T \hspace{-1pt} - \hspace{-1pt} t \hspace{-1pt} + \hspace{-1pt}1)$-th iterated function in \eqref{n_iterated_function}. We prove \eqref{interim_gen_info_bounds_1} by induction. At time $T$, recall from \eqref{eq_DP_memory_opt} that $J_{T}(m_T; T) = \inf_{u_T \in \mathcal{U}} \sup_{a_T, c_T \in [[A_T, C_T|m_T, u_T]]}(a_T $ $+ \gamma^T {\cdot} c_T)$. Using the indicator function in the RHS, $\sup_{a_T, c_T \in [[A_T, C_T|m_T, u_T]]}(a_T + \gamma^T {\cdot} c_T)
= \sup_{a_T \in \mathcal{A}, c_T \in \mathcal{C}}$ 
$ (a_T + \gamma^T {\cdot} c_T + \mathbb{I}(a_T, c_T\,|\, m_T, u_T))
= \sup_{a_T \in \mathcal{A}, c_T \in \mathcal{C}} \big(a_T + \gamma^T  {\cdot} c_T + \mathbb{I}(a_T, c_T\,|\, m_T, u_T)
- \sup_{a_T \in \mathcal{A}}(a_T + \mathbb{I}(a_T \,|\, m_T,$ $ u_T) ) \big) + \sup_{a_T \in [[A_T|m_T]]}$ 
$= \gamma^T {\cdot} \sup_{c_T \in \mathcal{C}} \big( c_T + \gamma^{-T} {\cdot} r_T(c_T$ 
$|\,m_T,u_T) \big) + \sup_{a_T \in [[A_T|m_T]]}a_T,$
where, in the second equality, we add and subtract the term $\sup_{a_T \in \mathcal{A}}(a_T + \mathbb{I}(a_T \,|\, m_T, u_T) )$; and, in the last equality, we use the definition of the accrued cost. Consequently, we can write that $J_T(m_T; T) 
= \inf_{u_T \in \mathcal{U}} \sup_{c_T \in \mathcal{C}, s_{T+1} \in \mathcal{S}}$ $ \gamma^T {\cdot} (c_T + \gamma {\cdot} \Lambda^{{{0}}}(s_{T+1}, \gamma^{T+1}) + \gamma^{-T} {\cdot} r_T(c_T, s_{T+1}\,|\,m_T, u_T) ) + \sup_{a_T \in [[A_T|m_T]]} a_T 
= \gamma^T {\cdot} \Lambda^1(\sigma_T(m_T), \gamma^T) + \sup_{a_T \in [[A_T|m_T]]}a_T$,
where, in the first equality, recall that $\Lambda^0(s_{T+1}, \gamma^{T+1}) :=0$ identically; % and that $\sup_{s_{T+1} \in \mathcal{S}}r_T(c_T, s_{T+1}|m_T, u_T) = r_T(c_T|m_T, u_T)$; 
and in the second equality, we use $r_T(c_T, s_{T+1}|m_T, u_T) = \rho(c_T, s_{T+1}|m_T, u_T)$ and the definition of $\Lambda^1(\sigma_T(m_T), \gamma^T)$ in \eqref{n_iterated_function}. This forms the basis of our induction. Next, consider the hypothesis that \eqref{interim_gen_info_bounds_1} holds for time $t+1$. Using the definition of the finite-horizon function at time $t$ and the induction hypothesis, $J_{t}(m_t;T) 
= \inf_{u_t \in \mathcal{U}} \sup_{m_{t+1} \in [[M_{t+1}|m_t, u_t]]} J_{t+1}(m_{t+1};T) 
= \inf_{u_t \in \mathcal{U}} \sup_{m_{t+1} \in [[M_{t+1}|m_t, u_t]]} \sup_{a_{t+1} \in [[A_{t+1}|m_{t+1}]]}\big(\gamma^{t+1}$ ${\cdot} \Lambda^{T-t}(\sigma_{t+1}(m_{t+1}), \gamma^{t+1}) + a_{t+1}) 
= \inf_{u_t \in \mathcal{U}} \sup_{m_{t+1}, a_{t+1} \in [[M_{t+1}, A_{t+1}|m_t, u_t]]}(\gamma^{t+1} {\cdot} \Lambda^{T-t}(\sigma_{t+1}($ $m_{t+1}), \gamma^{t+1})+ \gamma^t {\cdot} c_t + a_t) 
= \gamma^t {\cdot} V^{T-t+1}(\sigma_t(m_t), \gamma^t) +  \sup_{a_t \in [[A_t|m_t]]} a_t$, where, in the third equality, we use $a_{t+1} $ $ =a_t + \gamma^t {\cdot} c_t$ and rearrange the terms, and the fourth equality follows from the same sequence of arguments as time step $T$. 
This proves \eqref{interim_gen_info_bounds_1} by induction.
Then, \eqref{eq_gen_info_bounds} follows directly for all ${{t \in \mathbb{N}}}$ and all ${{n \in \mathbb{N}}}$ by substituting \eqref{interim_gen_info_bounds_1} into \eqref{eq_finite_DP_memory_opt} and selecting the horizon $T = t + n -1$.
%\begin{multline}
%    = \min_{u_t \in \mathcal{U}} \sup_{m_{t+1} \in \mathcal{M}_{t+1}, a_{t} \in \mathcal{A}, c_t \in \mathcal{C}_t}(\gamma^{t+1} {\cdot} \Lambda^{T-t}(\sigma_{t+1}( \\m_{t+1}))
%    + \gamma^t {\cdot} c_t + a_t + \mathbb{I}(c_t, \sigma_{t+1}(m_{t+1}), a_t|m_t, u_t)) \\
%    = \min_{u_t \in \mathcal{U}} \sup_{m_{t+1} \in \mathcal{M}_{t+1}, c_t \in \mathcal{C}_t}(\gamma^{t+1} {\cdot} \Lambda^{T-t}(\sigma_{t+1}( \\ m_{t+1}))
%    + \gamma^t {\cdot} c_t + r_t(c_t, \sigma_{t+1}(m_{t+1})|m_t, u_t)) + \sup_{a_t \in [[A_t|m_t]]} a_t \\
%    = \gamma^{t} {\cdot}\min_{u_t \in \mathcal{U}} \sup_{\sigma_{t+1}(m_{t+1}) \in \mathcal{S}, c_t \in \mathcal{C}_t}(\gamma {\cdot} \Lambda^{T-t}(\sigma_{t+1}(m_{t+1})) \\
%    + c_t + \frac{r_t(c_t, \sigma_{t+1}(m_{t+1})|\sigma_t(m_t), u_t)}{\gamma^t} ) + \sup_{a_t \in [[A_t|m_t]]} a_t \\
%    = \gamma^t {\cdot} V^{(T-t+1)}(\sigma_t(m_t), \gamma^t) +  \sup_{a_t \in [[A_t|m_t]]} a_t.
%\end{multline}
\end{proof}

%Then, we can prove the following bounds for any ${{t \in \mathbb{N}}}$, it holds that $\frac{\gamma^{T+1}{\cdot} c_{\min}}{1-\gamma} + \gamma^t {\cdot} \Lambda^{(T-t+1)}(\sigma_t(m_t), \gamma^t) + \sup_{a_t \in [[A_t|m_t]]}a_t \leq V_t(m_t) 
%        \leq \gamma^t {\cdot} \Lambda^{(T-t+1)}(\sigma_t(m_t), \gamma^t) + \sup_{a_t \in [[A_t|m_t]]}a_t + \frac{\gamma^{T+1}{\cdot} c_{\max}}{1-\gamma}.$
%This directly implies that 
%\begin{multline}
%    \gamma^t {\cdot} \lim_{T \to \infty}[\mathcal{T}^{T-t+1} \Lambda](\sigma_t(m_t), \gamma^t) + \sup_{a_t \in [[A_t|m_t]]}a_t  \\
%    = \gamma^t {\cdot} \Lambda^{{{\infty}}}(\sigma_t(m_t), \gamma^t) + \sup_{a_t \in [[A_t|m_t]]}a_t = V_t(m_t),
%\end{multline}
%or in words, the fixed point of the value iteration equations computes the optimal value functions of Problem \ref{problem_1}. Furthermore, it holds by substituting $t=0$ that $\sup_{y_0 \in \mathcal{Y}} \Lambda^{{{\infty}}}(\sigma_0(y_0), 1)  = \sup_{y_0 \in \mathcal{Y}} V_0(y_0) = \inf_{\boldsymbol{g} \in \mathcal{G}}\mathcal{J}(\boldsymbol{g})$. Furthermore, the optimal control strategy which minimizes the time-invariant fixed point is also an optimal solution to the original problem at each $t$.

Theorem \ref{thm_g_opt} allows us to characterize the error between the optimal value $V_0(y_0)$ and $\Lambda^n(\sigma_0(y_0), 1)$ for any $y_0 \in \mathcal{Y}$ by selecting $t=0$ in \eqref{eq_gen_info_bounds}, it follows that
$\frac{\gamma^{n} {\cdot} c^{\min}}{1-\gamma} + \Lambda^{{{n}}}(\sigma_0(y_0), 1) \leq V_0(y_0) \leq \Lambda^{{{n}}}(\sigma_0(y_0), 1)  + \frac{\gamma^{n} {\cdot} c^{\max}}{1-\gamma}$ (recall that $a_0 = 0$ and $z_0 = 1$). These bounds imply that as the number of iterations $n \to \infty$, the fixed point $\Lambda^{\infty}$ satisfies
\begin{gather} \label{eq_optimality_of_infor_state}
    \Lambda^{\infty}(\sigma_0(y_0), 1) = V_0(y_0).
\end{gather}
Thus, the DP in \eqref{n_iterated_function} recursively computes the optimal value of Problem \ref{problem_1}. Next, consider that the infimum is achieved in the RHS of \eqref{eq_general_value_operator} for each function $\Lambda^n$, ${{n \in \mathbb{N}}}$. Then, we define a time-invariant control strategy $\boldsymbol{\pi}^* := (\pi^*, \pi^*,\dots)$ where the control law at each $t \in \mathbb{N}$ is the minimizing argument for $\Lambda^{\infty}$, i.e., $\pi^*(s, z) := \arg \min_{u \in \mathcal{U}} \sup_{c \in \mathcal{C}, s' \in \mathcal{S}}(c + \gamma {\cdot} \Lambda^{\infty}(s', \gamma {\cdot}z) + \frac{\rho(c, s'\,|\,s, u)}{z})$ for all $s \in \mathcal{S}$ and $z \in (0,1]$. 
A corresponding memory-based strategy is $\boldsymbol{g}^* = (g_0^*, g_1^*, \dots)$ with $g_t^*(m_t) := \pi^*(\sigma_t(m_t), \gamma^t)$ for all $m_t \in \mathcal{M}_t$ and all $t$. Then, $\boldsymbol{g}^*$ achieves the optimal value in Problem \ref{problem_1}, i.e., $V^{\boldsymbol{g}^*}_0(y_0) = V_0(y_0)$ (see proof in Appendix B). Thus, the information-state gives an optimal solution to Problem \ref{problem_1}.

\subsection{Examples of Information States} \label{subsection:info_state_examples}

In this subsection, we consider a system with a known state-space model to present specific information states which satisfy Definition \ref{def_information_state}. Consider a system with a state $X_t \in \mathcal{X}$ which starts at $X_0$ and evolves as $X_{t+1} = f_t(X_t, U_t, W_t)$, the observation is  $Y_t = h_t(X_t, W_t)$ and the incurred cost is $C_t = d_t(X_t,U_t)$ at each $t$, where $N_t \in \mathcal{N}$ is an uncontrolled disturbance. The uncontrolled variables $\{X_0, N_t, W_t \,|\, {{t \in \mathbb{N}}}\}$ take realizations independently. Next, we give examples of information states for different cases:

\textit{1) Perfectly observed systems:} For all $t \in \mathbb{N}$, let $Y_t = X_t$, an information state is $S_t \hspace{-1pt} = \hspace{-1pt}  X_t \hspace{-1pt} \in \hspace{-1pt} \mathcal{X}$.

\textit{2) Perfectly observed systems with deep dynamics:} For all $t \in \mathbb{N}$, let $Y_t = X_t$ and %the state evolves as a function of $k + 1 \in \mathbb{N}$ previous states, i.e., 
$X_{t+1} = f(X_{t:t-k}, U_t, W_t)$, an information state is $S_t = (X_{t-k}, \dots, X_t) \in \mathcal{X}^{k+1}$. 

\textit{3) Partially observed systems:} Consider a generic partially observed system. An information state at each $t \in \mathbb{N}$ is the function-valued variable $S_t: \mathcal{X} \to \{-\infty\} \cup [-a^{\max}, 0]$. %Note that this takes place in a time-homogenous function space. 
At each $t$, for a given $m_t \in \mathcal{M}_t$, its realization is a function $s_t(x_t) := r_t(x_t|m_t)$, where $r_t(\cdot)$ is an accrued distribution. This is a normalization of the results in \cite{james1994risk,  bernhard2003minimax}.

\textit{4) Systems with action dependent costs:} For all $t \in \mathbb{N}$, let the cost of a partially observed system be $d_t(U_t) \in \mathbb{R}_{\geq0}$. Then, an information state is the conditional range $S_t = [[X_t$ $|\;M_t]] \in \mathcal{B}(\mathcal{X})$, where $\mathcal{B}(\mathcal{X})$ is the set of all subsets of $\mathcal{X}$. 

\begin{remark}
    Definition \ref{def_information_state} helps us identify information states when system dynamics are known. However, such a representation often needs to be learned purely from observation and cost data, without knowledge of dynamics. Thus, in the next section, we specialize the notion of information states to systems with observable costs and define approximate information states that can be learned from output data.
\end{remark}

\section{Systems with Observable Costs} \label{section:perfectly_observed}

In this section, we analyze Problem \ref{problem_1} in the case where the agent observes the incurred cost at each instance of time. Thus, at each ${{t \in \mathbb{N}}}$, the agent receives a realization of $(Y_t, C_t)$ and the memory is $M_t = (Y_{0:t}, C_{0:t-1}, U_{0:t-1})$. We first prove that for such a system, the accrued distribution in Definition \ref{def_information_state} at each ${{t \in \mathbb{N}}}$ reduces to simply an indicator.

\begin{lemma} \label{lemma_specialization}
    Consider Problem \ref{problem_1} with observable costs. At each ${{t \in \mathbb{N}}}$, for any given realizations $c_t \in \mathcal{C}$, $m_{t+1} \in \mathcal{M}_{t+1}$, $m_t \in \mathcal{M}_t$, and $u_t \in \mathcal{U}$, it holds that
    \begin{gather} \label{eq_lemma_specialization}
        r_t(c_t, m_{t+1}\,|\,m_t, u_t) = \mathbb{I}(c_t, m_{t+1}\,|\,m_t, u_t).
    \end{gather}
\end{lemma}

\begin{proof}
Let the given realization of the memory at time $t$ be $m_t = (\Tilde{y}_{0:t}, \Tilde{c}_{0:t-1}, \Tilde{u}_{0:t-1})$. Then, we expand the accrued distribution at any $t$ as $r_t(c_t, m_{t+1}|m_t, u_t) 
= \sup_{a_t \in \mathcal{A}} \big( a_t + \mathbb{I}(a_t, c_t, m_{t+1}|m_t, u_t) \big)$ $- \sup_{a_t \in \mathcal{A}} \big( a_t + \mathbb{I}(a_t|m_t)\big)$ 
$= \sup_{a_t \in \mathcal{A}} \big( a_t + \mathbb{I}(a_t|m_t, u_t, c_t,$ $ m_{t+1}) + \mathbb{I}(c_t, m_{t+1}|m_t, u_t) \big) - \sum_{\ell=0}^{t-1} \gamma^{\ell} {\cdot} \Tilde{c}_{\ell}$ 
$= \sum_{\ell=0}^{t-1} \gamma^{\ell} {\cdot} \Tilde{c}_{\ell} + \mathbb{I}(c_t, m_{t+1}|m_t, u_t) - \sum_{\ell=0}^{t-1} \gamma^{\ell} {\cdot} \Tilde{c}_{\ell}$ $= \mathbb{I}(c_t, m_{t+1}|m_t, u_t)$, where, in the second equality, the realization of $A_t$ is determined as $\tilde{a}_t = \sum_{\ell=0}^{t-1} \gamma^{\ell} {\cdot} \tilde{c}_{\ell}$ given $m_t$; and in the third equality, $\mathbb{I}(a_t|m_t, u_t, c_t, m_{t+1}) = 0$ only if $a_t = \tilde{a}_t$.
%    \begin{multline}
%        r_t(c_t, m_{t+1}|m_t, u_t) = \sup_{a_t \in \mathcal{A}} \big( a_t + \mathbb{I}(a_t, c_t, m_{t+1}|m_t, u_t) \big) \\
%        - \sup_{a_t \in \mathcal{A}} \big( a_t + \mathbb{I}(a_t|m_t)\big)\\
%        = \sup_{a_t \in \mathcal{A}} \big( a_t + \mathbb{I}(a_t|m_t, u_t, c_t, m_{t+1}) \\
%        + \mathbb{I}(c_t, m_{t+1}|m_t, u_t) \big) - \sum_{\ell=0}^{t-1} \gamma^{\ell} {\cdot} \Tilde{c}_{\ell} \\
%        = \sum_{\ell=0}^{t-1} \gamma^{\ell} {\cdot} \Tilde{c}_{\ell} + \mathbb{I}(c_t, m_{t+1}|m_t, u_t) - \sum_{\ell=0}^{t-1} \gamma^{\ell} {\cdot} \Tilde{c}_{\ell} \\
%        = \mathbb{I}(c_t, m_{t+1}|m_t, u_t).
%    \end{multline}
\end{proof}

Next, we present a simpler notion of information states.

%a simpler definition for information states in systems with observable costs.

\begin{definition} \label{def_info_specialized}
    An \textit{information state} for Problem \ref{problem_1} with observable costs at any ${{t \in \mathbb{N}}}$ is an uncertain variable $\bar{S}_t = \bar{\sigma}_t(M_t)$ taking values in a bounded, time-invariant set $\bar{\mathcal{S}}$. For all $t \in \mathbb{N}$, for all $m_t \in \mathcal{M}_t$ and $u_t \in \mathcal{U}_t$, it satisfies that
    \begin{gather} \label{eq_def_info_specialized}
        [[C_t, \bar{S}_{t+1} \, | \,m_t, u_t]] = [[C_t, \bar{S}_{t+1}\, | \,\bar{\sigma}_t(m_t), u_t]]. 
    \end{gather}
\end{definition}

Next, we use the information state from Definition \ref{def_info_specialized} to construct a time-invariant operator $\bar{\mathcal{T}} : [\bar{\mathcal{S}} \to \mathbb{R}] \to [\bar{\mathcal{S}} \to \mathbb{R}]$, such that, for any uniformly bounded function $\bar{\Lambda}: \bar{\mathcal{S}} \to \mathbb{R}$,
    \begin{gather} \label{eq_info_value_operator}
        [\bar{\mathcal{T}} \bar{\Lambda}](\bar{s}) := \inf_{u \in \mathcal{U}} \sup_{c, \bar{s}' \in [[C, \bar{S}'|\bar{s}, u]]} \big(c + \gamma {\cdot} \bar{\Lambda}(\bar{s}') \big).
    \end{gather}
Note that \eqref{eq_info_value_operator} is simpler than \eqref{eq_general_value_operator} and $\bar{\mathcal{T}}$ is also a contraction mapping. Thus, $\bar{\Lambda} = \bar{\mathcal{T}} \bar{\Lambda}$ admits a unique solution $\bar{\Lambda}^{\infty} = \bar{\mathcal{T}} \bar{\Lambda}^{\infty}$. Starting with $\bar{\Lambda}^0(\bar{s}) := 0$, the fixed-point iteration around $\bar{\mathcal{T}}$ generates a sequence of functions
    \begin{gather} \label{eq_specialized_iteration}
        \bar{\Lambda}^{n+1}(\bar{s}) = [\bar{\mathcal{T}} \bar{\Lambda}^n](\bar{s}) = [\bar{\mathcal{T}}^n \bar{\Lambda}^0](\bar{s}),
    \end{gather}
for all $n=1,2,\dots$, such that $\lim_{n \to \infty} \bar{\mathcal{T}}^n \bar{\Lambda}^0 = \bar{\Lambda}^{\infty}$. Next, we establish error bounds for using $\bar{\Lambda}^n(\bar{\sigma}_t(m_t))$, ${{n \in \mathbb{N}}}$, to estimate $V_t(m_t)$ for all $t$ in Problem \ref{problem_1} with observable costs.

\begin{theorem} \label{thm_specialized_bounds}
    Consider the function $\bar{\Lambda}^n$ generated using \eqref{eq_specialized_iteration} for any $n \in \mathbb{N}$. Then, for all ${{t \in \mathbb{N}}}$, it holds that
    \begin{multline} \label{eq_thm_specialized_bounds}
        \frac{\gamma^{n+t}{\cdot} c^{\min}}{1-\gamma} + \gamma^t {\cdot} \bar{\Lambda}^{{{n}}}(\bar{\sigma}_t(m_t)) + \sup_{a_t \in [[A_t|m_t]]}a_t  \; \leq \; V_t(m_t) \\ 
        \leq \sup_{a_t \in [[A_t|m_t]]}a_t + \gamma^t {\cdot} \bar{\Lambda}^{{{n}}}(\bar{\sigma}_t(m_t))  + \frac{\gamma^{n+t} {\cdot} c^{\max}}{1-\gamma}.
    \end{multline}
\end{theorem}

\begin{proof}
    We show \eqref{eq_thm_specialized_bounds} by combining arguments in Theorem \ref{thm_g_opt} with \eqref{eq_lemma_specialization} from Lemma \ref{lemma_specialization}. Thus, we first show that for any horizon $T \in \mathbb{N}$, the following holds for each $t=0,\dots,T$:
    \begin{gather} \label{second_interim}
        J_{t}(m_t; T) = \gamma^t {\cdot} \bar{\Lambda}^{T-t+1}\big(\bar{\sigma}_t(m_t)\big) + \sup_{a_t \in [[A_t|m_t]]}a_t.
    \end{gather}
    We can prove \eqref{second_interim} by induction. At time $T$, using the definition of the finite-horizon function $J_T(m_T; T) 
    = \inf_{u_T \in \mathcal{U}} \sup_{a_T, c_T \in [[A_T, C_T|m_T, u_T]]}(a_T + \gamma^T {\cdot} c_T) = \inf_{u_T \in \mathcal{U}} \sup_{c_T \in [[C_T|m_T, u_T]]}c_T + \sup_{a_T \in [[A_T|m_T]]} a_T 
    = \inf_{u_T \in \mathcal{U}} \sup_{c_T, \bar{\sigma}_{T+1}(m_{T+1}) \in [[C_T, S_{T+1}|m_T, u_T]]}(c_T + \gamma^T {\cdot} \bar{\Lambda}^0(\bar{\sigma}_{T+1}(m_{T+1}))) + \sup_{a_T \in [[A_T|m_T]]} a_T 
    = \bar{\Lambda}^1(\sigma_T(m_T)) + \sup_{a_T \in [[A_T|m_T]]} a_T$, 
    where, in the second equality, note that $A_T$ is completely determined given $M_T$ as in Lemma \ref{lemma_specialization}; and in the third equality,  we note that $\bar{\Lambda}^0(\bar{\sigma}_{T+1}(m_{T+1})) = 0$.
    %; and in the fourth equality, we use the definition of $\bar{\Lambda}^1$. 
    This forms the basis of our induction. Next, consider as a hypothesis that \eqref{second_interim} holds at time $t+1$. Using the definition of the finite-horizon function at time $t$, $J_t(m_t; T) 
    = \inf_{u_t \in \mathcal{U}} \sup_{m_{t+1} \in [[M_{t+1}|m_t, u_t]]}\sup_{c_t, a_{t}\in [[C_t, A_{t}|m_{t+1}]]}(\gamma^{t+1} {\cdot}$ $ \bar{\Lambda}^{T-t}(\bar{\sigma}_{t+1}(m_{t+1})) + \gamma^t {\cdot} c_t + a_t)
    = \inf_{u_t \in \mathcal{U}}$ $ \sup_{c_t, a_t, m_{t+1} \in [[C_t, A_t, M_{t+1}|m_t, u_t]]}(\gamma^{t+1} {\cdot} \bar{\Lambda}^{T-t}(\bar{\sigma}_{t+1}(m_{t+1}))$ $ + \gamma^t {\cdot} c_t + a_t) 
    = \inf_{u_t \in \mathcal{U}}\sup_{c_t, \bar{\sigma}_{t+1}(m_{t+1}) \in [[C_t, \bar{S}_{t+1}|m_t, u_t]]}$ $(\gamma^{t+1} {\cdot} \bar{\Lambda}^{T-t}(\bar{\sigma}_{t+1}(m_{t+1})) +\gamma^t{\cdot}c_t ) + \sup_{a_t \in [[A_t|m_t]]}a_t$
    $= \inf_{u_t \in \mathcal{U}}\sup_{c_t, \bar{\sigma}_{t+1}(m_{t+1}) \in [[C_t, \bar{S}_{t+1}|\bar{\sigma}_t(m_t), u_t]]}(\gamma^{t+1} {\cdot} \bar{\Lambda}^{T-t}($ $\bar{\sigma}_{t+1}(m_{t+1})) + \gamma^t {\cdot} c_t ) + \sup_{a_t \in [[A_t|m_t]]}a_t 
    = \gamma^t \cdot \bar{\Lambda}^{T-t+1}(\bar{\sigma}_{t}(m_t)) + \sup_{a_t \in [[A_t|m_t]]}a_t$, where, in the third equality, we use the same arguments as in Lemma \ref{lemma_specialization}; in the fourth equality, we use \eqref{eq_def_info_specialized} from Definition \ref{def_info_specialized}; and in the last equality, we use the definition of $\bar{\Lambda}^{T-t+1}$ from \eqref{eq_specialized_iteration}. This proves \eqref{second_interim} using induction.
    Then, \eqref{eq_thm_specialized_bounds} follows directly for all ${{t \in \mathbb{N}}}$ and all ${{n \in \mathbb{N}}}$ by substituting \eqref{second_interim} into \eqref{eq_finite_DP_memory_opt} and selecting a horizon $T = t + n -1$.
\end{proof}

In \eqref{eq_thm_specialized_bounds}, we select $t=0$ and let $n \to \infty$ to establish that $\bar{\Lambda}^{\infty}(\bar{\sigma}_0(y_0)) = V_0(y_0)$. Thus, when Problem \ref{problem_1} has observable costs,
the fixed point $\bar{\Lambda}^{\infty}$ computes the optimal value function $V_0$ as a direct consequence of Theorem \ref{thm_specialized_bounds}. Next, consider that the infimum is achieved in the RHS of $[\bar{\mathcal{T}}\bar{\Lambda}^{n}](\bar{s})$ for all $\bar{s} \in \bar{\mathcal{S}}$ and ${{n \in \mathbb{N}}}$. We define a strategy $\boldsymbol{\pi}^* = (\pi^*, \pi^*, \dots)$, where $\bar{\pi}^*: \bar{\mathcal{S}} \to \mathcal{U}$ is the minimizing argument in RHS of \eqref{eq_info_value_operator} for $\Lambda = \bar{\Lambda}^{\infty}$. Then, from the same arguments as in Subsection \ref{subsection:basic_info_states}, it holds that the memory-based strategy $\bar{\boldsymbol{g}}^* = (\bar{g}_0^*, \bar{g}_1^*, \dots)$, where $\bar{g}_t^* := \bar{\pi}^*(\sigma_t(m_t))$, gives an optimal solution to Problem \ref{problem_1} with observable costs.

\begin{remark}
    If a partially observed system with observable costs has the state-space model from Subsection \ref{subsection:info_state_examples}, an information state at each $t$ is $\bar{S}_t = [[X_t|M_t]] \in \mathcal{B}(\mathcal{X})$. This is simpler than the accrued cost function in Subsection \ref{subsection:info_state_examples}.
\end{remark}

\begin{remark}
    When attempting to learn an information state that satisfies Definition \ref{def_info_specialized} using only output data, we may not be able to satisfy \eqref{eq_def_info_specialized} exactly. Thus, in Subsection \ref{subsection:approx}, we relax this definition for approximate information states.
\end{remark}

\subsection{Approximate Information States} \label{subsection:approx}

In this subsection, we define approximate information states that approximately satisfy \eqref{eq_def_info_specialized}, and construct a time-invariant approximate DP of Problem \ref{problem_1} using them. Then, we bound the resulting error estimating the optimal value and the performance loss of the resulting approximate strategy.

\begin{definition} \label{def_approximate_info}
    An \textit{approximate information state} for Problem \ref{problem_1} with observable costs at any ${{t \in \mathbb{N}}}$ is an uncertain variable $\hat{S}_t = \hat{\sigma}_t(M_t)$ taking values in a bounded, time-invariant set $\hat{\mathcal{S}}$. Furthermore, there exists a parameter $\epsilon \in \mathbb{R}_{\geq0}$ such that for all $m_t \in \mathcal{M}_t$ and $u_t \in \mathcal{U}$ and $t \in \mathbb{N}$, it satisfies
    \begin{gather} \label{eq_def_approximate_info}
        \mathcal{H}\big( [[C_t, \hat{S}_{t+1}|m_t, u_t]], [[C_t, \hat{S}_{t+1}|\hat{\sigma}_t(m_t), u_t]] \big) \leq \epsilon,
    \end{gather}
    where recall that $\mathcal{H}$ is the Hausdorff distance defined in \eqref{H_met_def}.
\end{definition}

To compute an approximate value and an approximate control strategy, we proceed with approximate information states just as we did with information states. First, we construct a time-invariant operator $\hat{\mathcal{T}}: [\hat{\mathcal{S}} \to \mathbb{R}] \to [\hat{\mathcal{S}} \to \mathbb{R}]$, such that for any uniformly bounded function $\hat{\Lambda}: \hat{\mathcal{S}} \to \mathbb{R}$,
\begin{multline} \label{eq_DP_approx}
    [\hat{\mathcal{T}} \hat{\Lambda}](\hat{s}) := \inf_{u \in \mathcal{U}} \sup_{c, \hat{s}' \in [[C, \hat{S}'|\hat{s}, u]]}\big(c + \gamma {\cdot} \hat{\Lambda}(\hat{s}')\big).
\end{multline}
%Due to discounting in the RHS of \eqref{eq_DP_approx}, 
Note that
$\hat{\mathcal{T}}$ is a contraction mapping and thus, the equation $\hat{\Lambda} = \hat{\mathcal{T}} \hat{\Lambda}$ admits a unique solution $\hat{\Lambda}^{\infty} = \hat{\mathcal{T}} \hat{\Lambda}^{\infty}$. Then, starting with $\hat{\Lambda}^0(\hat{s}):= 0$ the fixed-point iteration around $\hat{\mathcal{T}}$ recursively generates the functions $\hat{\Lambda}^{n+1}(\hat{s}) = [\hat{\mathcal{T}} \hat{\Lambda}^n](\hat{s}) = [\hat{\mathcal{T}}^n \hat{\Lambda}^0](\hat{s})$,
for all $n=1,2,\dots$, such that $\lim_{n \to \infty} \hat{\mathcal{T}}^n \hat{\Lambda}^0 = \hat{\Lambda}^{\infty}$. This forms our approximate DP decomposition. Next, consider that the infimum is achieved in the RHS of $[\hat{\mathcal{T}} \hat{\Lambda}^{n}](\hat{s})$ for all $\hat{s} \in \hat{\mathcal{S}}$ and all $n \in \mathbb{N}$. We an approximate strategy $\hat{\boldsymbol{\pi}}^* = (\hat{\pi}^*, \hat{\pi}^*, \dots)$, where $\hat{\pi}^*: \hat{S} \to \mathcal{U}$ is the minimizing argument in the RHS of \eqref{eq_DP_approx} for $\hat{\Lambda} = \hat{\Lambda}^{\infty}$. Then, a corresponding memory-based strategy is $\hat{\boldsymbol{g}}^* := (\hat{g}_0^*, \hat{g}_1^*, \dots)$ with $\hat{g}_t^* := \pi^*(\sigma_t(m_t))$ for all $t\in\mathbb{N}$. Next, we bound both the approximation error between the optimal value $V_0(y_0)$ and $\hat{\Lambda}^{\infty}(\hat{\sigma}_0(y_0))$, and the performance loss when implementing $\hat{\boldsymbol{g}}^*$ to generate the control actions.

%After bounding the value approximation error, we also seek to bound the maximum performance loss in the implementation of an approximately optimal strategy. Consider an approximate strategy $\boldsymbol{\pi}^* = (\pi^*, \pi^*, \dots)$, where $\pi^*$ is the argument that minimizes the RHS in \eqref{eq_DP_approx}. Then, a corresponding control strategy is $\boldsymbol{g}^* = (g^*_0, g^*_1, \dots)$ such that $g^*_t(m_t) = \pi^*(\hat{\sigma}_t(m_t))$. 

\begin{theorem} \label{thm_approx}
    Let the functions $\hat{\Lambda}^n$ be Lipschitz continuous with a constant $L_{\hat{\Lambda}} \hspace{-1pt}\in\hspace{-1pt} \mathbb{R}_{\geq0}$ for all ${{n \hspace{-1pt}\in\hspace{-1pt} \mathbb{N}}}$. Then, we have that
    \begin{align} \label{eq_thm_approx}
        \textbf{a) } \; &|V_0(y_0) - \hat{\Lambda}^{\infty}(\hat{\sigma}_0(y_0))| \leq {\hat{L} {\cdot} \epsilon}{\cdot}{(1-\gamma)^{-1}}, \\
        \textbf{b) } \; &|V_0(y_0) - V_0^{\hat{\boldsymbol{g}}^*}(y_0)| \leq {2{\cdot}\hat{L}{\cdot}\epsilon}{\cdot}{(1-\gamma)^{-1}}, \label{eq_thm_approx_g}
    \end{align}
    where $\hat{L} = \max\{\gamma {\cdot} L_{\hat{\Lambda}},1\}$. 
\end{theorem}

\begin{proof}
We show \eqref{eq_thm_approx} using \eqref{eq_finite_DP_memory_opt} from Lemma \ref{finite_DP_memory}. Thus, we first show that for any $T \hspace{-1pt} \in \hspace{-1pt} \mathbb{N}$, it holds for all $t = 0,\dots,T$: \hspace{-10pt}
    \begin{gather} \label{eq_interim_third}
         \hspace{-5pt} |J_{t}(m_t; \hspace{-1pt} T) \hspace{-2pt} - \hspace{-1pt} \gamma^t {\cdot} \hat{\Lambda}^{T-t+1} \hspace{-1pt} \big(\hat{\sigma}_t(m_t)\big) \hspace{-2pt}  - \hspace{-7pt} \sup_{a_t \in [[A_t|m_t]]} \hspace{-7pt} a_t| \hspace{-1pt} \leq \hspace{-1pt} \beta_{t}(T), \hspace{-2pt}
    \end{gather}
where $\beta_{t}(T) = \beta_{t+1}(T) + \gamma^t {\cdot} \hat{L} {\cdot} \epsilon$ and $\beta_{T}(T) = \gamma^T {\cdot} \hat{L} {\cdot} \epsilon$. %with $\hat{L} = \max\{\gamma {\cdot} L_{\hat{\Lambda}},1\}$. 
We prove \eqref{eq_interim_third} by induction. At time $T$, recall that $J_{T}(m_T; T) = \inf_{u_T \in \mathcal{U}}\sup_{c_T \in [[C_T|m_T, u_T]]} \gamma^T {\cdot}c_T + \sup_{a_T \in [[A_T|m_T]]}$ $ a_T$ using the arguments in Lemma \ref{lemma_specialization}. This implies that $|J_{T}(m_T; T)- \gamma^T {\cdot} \hat{\Lambda}^1\big(\hat{\sigma}_T(m_T)\big) - \sup_{a_T \in [[A_T|m_T]]} a_T| 
= |\inf_{u_T \in \mathcal{U}} \sup_{c_T \in [[C_T|m_T, u_T]]} \gamma^T {\cdot}c_T - \gamma^T {\cdot} \hat{\Lambda}^1\big(\hat{\sigma}_T(m_T)\big)| 
= \gamma^T {\cdot} |\inf_{u_T \in \mathcal{U}} \sup_{c_T, \hat{s}_{T+1} \in [[C_T, \hat{S}_{T+1}|m_T, u_T]]} (c_T + \gamma {\cdot} \hat{\Lambda}^0(\hat{s}_{T+1}))- \inf_{u_T \in \mathcal{U}} \sup_{c_T, \hat{s}_{T+1} \in [[C_T, \hat{S}_{T+1}|\hat{\sigma}_T(m_T), u_T]]}$ 
$(c_T + \gamma {\cdot} \hat{\Lambda}^0 \hspace{-1pt} (\hat{s}_{T+1} \hspace{-1pt}) \hspace{-1pt}) \hspace{-1pt}| \hspace{-1pt} \leq \hspace{-1pt} \gamma^T {\cdot} \hat{L} {\cdot} \sup_{u_T \in \mathcal{U}} \hspace{-2pt} \mathcal{H}([[C_T, \hat{S}_{T+1}|m_T, u_T]],$ $ [[C_T, \hat{S}_{T+1}|\hat{\sigma}_T(m_T), u_T]])
\leq \gamma^T {\cdot} \hat{L} {\cdot} \epsilon$, where, in the second equality, we note that $\Lambda^0(\hat{s}_{T+1}) = 0$ identically; in the first inequality, we note that $\hat{L} = \max\{\gamma {\cdot} L_{\hat{\Lambda}}, 1\}$ is the Lipschitz constant of $(c_T + \gamma {\cdot} \hat{\Lambda}^0(\hat{s}_{T+1}) )$ with respect to $(c_T, \hat{s}_{T+1})$ and use \eqref{H_met_property}; and in the second inequality, we use \eqref{eq_def_approximate_info}. This forms the basis of our induction. Next, we consider the hypothesis that \eqref{eq_thm_approx} holds at time $t+1$. Using the hypothesis and rearranging terms,
$J_{t+1}(m_{t+1};T) \leq \beta_{t+1}(T) + \gamma^{t+1} {\cdot} \hat{\Lambda}^{T-t}$ $ (\hat{\sigma}_{t+1}(m_{t+1})) + \sup_{a_{t+1} \in [[A_{t+1}|m_{t+1}]]} a_{t+1}$. Then, at time $t$, $|J_t(m_t; T) - \gamma^t {\cdot} \hat{\Lambda}^{T-t+1}\big(\hat{\sigma}_{t}(m_t)\big)- \sup_{a_t \in [[A_t|m_t]]}$ $ a_t|
\hspace{-1pt}
\leq \hspace{-1pt}\beta_{t+1}(T) \hspace{-1pt} + \hspace{-1pt} |\inf_{u_t \in \mathcal{U}} \sup_{m_{t+1} \in [[M_{t+1}|m_t, u_t]]}(\gamma^{t+1} {\cdot} \hat{\Lambda}^{T-t}$ $(\hat{\sigma}_{t+1}(m_{t+1}))+ \sup_{a_t, c_t \in [[A_t, C_t|m_{t+1}]]}$ $(a_t + \gamma^t {\cdot} c_t) ) - \gamma^t {\cdot} $ $\hat{\Lambda}^{T-t+1}(\hat{\sigma}_{t}(m_t)) - \sup_{a_t \in [[A_t|m_t]]}a_t |
\leq \beta_{t+1}(T) + \gamma^t {\cdot} \sup_{u_t \in \mathcal{U}}$ 
$|\sup_{c_t, \hat{\sigma}_{t+1}(m_t) \in [[C_t, \hat{S}_{t+1}|m_t, u_t]]} (c_t + \gamma {\cdot} \hat{\Lambda}^{T-t}(\hat{\sigma}_{t+1}($ $m_{t+1})))$ $- \sup_{c_t, \hat{s}_{t+1} \in [[C_t, \hat{S}_{t+1}|\hat{\sigma}_t(m_t), u_t]]}(c_t + \gamma {\cdot} \hat{\Lambda}^{T-t} (\hat{s}_{t+1}))| \leq \beta_{t+1}(T) + \gamma^t {\cdot} \hat{L} {\cdot} \epsilon$, where, in the second inequality, we use arguments in Lemma \ref{lemma_specialization} for $[[A_t, C_t|m_{t+1}]] = [[A_t|m_{t+1}]] \times [[C_t|m_{t+1}]]$; and, in the third inequality we use \eqref{H_met_property} and \eqref{eq_def_approximate_info}. This proves \eqref{eq_interim_third} for all $t$ using induction. 

Next, for the iterated function $\hat{\Lambda}^n$, we select a horizon $T = n-1$ and set $t=0$ in \eqref{eq_interim_third}, to write that
    $|J_0(y_0; T) - \hat{\Lambda}^n(\hat{\sigma}_0(y_0))| \leq \beta_0(T)$,
where $\beta_0(T) = \sum_{\ell=0}^{n-1} \gamma^{\ell} {\cdot} \hat{L} {\cdot} \epsilon$. As $n \to \infty$ with $T=n-1$, note that $\lim_{T\to \infty} J_0(y_0; T) = V_0(y_0)$, $\lim_{n \to \infty} \hat{\Lambda}^n(\hat{\sigma}_0(y_0)) = \hat{\Lambda}^\infty(\hat{\sigma}_0(y_0))$, and $\lim_{T \to \infty} \beta_0(T) = \frac{\hat{L} {\cdot} \epsilon}{1-\gamma}$. %This completes the proof for \eqref{eq_thm_approx}. 
The proof for \eqref{eq_thm_approx_g} follows from a similar series of arguments.
\end{proof}

\subsection{Alternate Characterization} \label{subsection:alternate}

When exploring whether an uncertain variable is a valid candidate to be considered for an approximate information state, it may be difficult to verify \eqref{eq_def_approximate_info}. Thus, we present two \textit{stronger} conditions that are easier to verify. To establish that $\hat{S}_t = \hat{\sigma}_t(M_t)$, $t \in \mathbb{N}$, satisfies \eqref{eq_def_approximate_info}, the following two conditions should hold (see proof in Appendix C):

\textit{1) State-like evolution:} There exists a Lipschitz continuous function $\psi:\hat{\mathcal{S}} \times \mathcal{U} \times \mathcal{Y} \to \mathcal{S}$, such that
\begin{gather}
    \hat{\sigma}_{t+1}(M_{t+1}) = \psi(\hat{\sigma}_t(M_t), U_t, Y_{t+1}). \label{ap2a}
\end{gather}

\textit{2) Sufficient to approximate outputs:} For all $m_t \in \mathcal{M}_t$ and $u_t \in \mathcal{U}$, there exists a constant $\delta \in \mathbb{R}_{\geq0}$ such that
\begin{gather}
    \mathcal{H}([[C_t, Y_{t+1}|m_t, u_t]]], [[C_t, Y_{t+1}|\hat{\sigma}_{t}(m_t), u_t]]) \leq \delta. \label{ap2b}
\end{gather}

\section{Numerical Example} \label{section:example}

We consider an agent pursuing a target across a $5 \times 5$ grid with obstacles. At each $t \in \mathbb{N}$, the agent's position $X_t^\text{ag}$ and the target's position $X_t^{\text{ta}}$ each take values in the set of grid cells $\mathcal{X} = \big\{(0,0),(0,1),$ $\dots,(4,4)\big\} \setminus \mathcal{O}$, where $\mathcal{O} \subset \mathcal{X}$ is the set of obstacles. Let $\mathcal{W} = \{(-1,0),(1,0),(0,0),(0,1),(0,-1)\}$, $\mathcal{N} = \{(0,-1),(0,0),(0,1)\}$, and $\mathcal{U} = \mathcal{W} \times \{\xi\}$, where $\xi$ denotes a ``stop'' action. Starting at $X_0^{\text{ta}} \in \mathcal{X}$, the target's position evolves as $X_{t+1}^{\text{ta}} = \delta(X_t^{\text{ta}} + W_t \in \mathcal{X}){\cdot}(X_t^{\text{ta}} + W_t) + (1 - \delta(X_t^{\text{ta}} + W_t \in \mathcal{X}) ){\cdot} X_t^{\text{ta}}$, where $W_t \in \mathcal{W}$ and $\delta$ is returns $1$ or $0$ after checking the argument. 
At each $t$, the agent observes their own position perfectly and the target's position as $Y_t = \delta(X_t^{\text{ta}} + N_t \in \mathcal{X}){\cdot}(X_t^{\text{ta}} + N_t) + (1-\delta(X_t^{\text{ta}} + N_t \in \mathcal{X}) ){\cdot} X_t^{\text{ta}}$, where $N_t \in \mathcal{N}$. 
Then, the agent selects an action $U_t \in \mathcal{U}$ which is either to move or to stop. If the agent moves, i.e., $U_t \neq \xi$, then $X_{t+1}^{\text{ag}} = \delta(X_t^{\text{ag}} + U_t \in \mathcal{X}){\cdot}(X_t^{\text{ag}} + U_t) + (1 - \delta(X_t^{\text{ag}} + U_t \in \mathcal{X}) ){\cdot} X_t^{\text{ag}}$. 
The agent incurs a cost $C_t = 2$. If the agent stops, i.e., $U_t = \xi$, they incur a terminal cost $10 {\cdot} \eta(X_T^{\text{ta}}, X_T^{\text{ag}})$ for the L$1$ distance from the target. 
We illustrate this pursuit problem in Fig. \ref{fig:1a}, where the black cells are obstacles, the red triangle is the agent, the blue circle is the observation, and the blue disk is the target.

\begin{figure}[t!]
  \centering
  %\begin{multicols}{2}
  %\captionsetup{justification=centering}
  \subfigure[The grid \label{fig:1a}]{\includegraphics[width=0.27\linewidth, keepaspectratio]{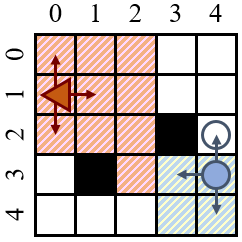}}
  \hspace{10pt}
  \subfigure[Encoder-decoder architecture \label{fig:1b}]{\includegraphics[width=0.63\linewidth, keepaspectratio]{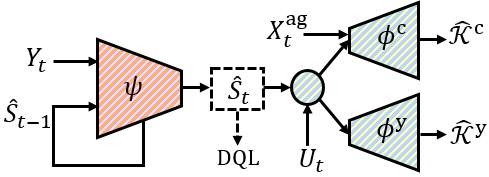}}
  %\end{multicols}
  \vspace{-6pt}
  \caption{The pursuit problem with $x_0^{\text{ag}} = (0,1)$, $x_0^{\text{ta}} = (4,3)$ and $y_0 = (4,2)$ is in (a). The neural network architecture for the AIS is in (b).}
  \label{fig:illustration}
  \vspace{-20pt}
\end{figure}

We consider the pursuit problem when the agent is aware of their own dynamics, but unaware of the observation model and target's dynamics.
Thus, we train an approximate information state (AIS) model to learn a representation of the target's dynamics using observations, actions, and incurred costs to enforce \eqref{ap2a} and \eqref{ap2b}. 
The AIS is generated by a neural networks in an encoder-decoder architecture, as shown in Fig. \ref{fig:1b}.
At each $t \in \mathbb{N}$, the \textit{encoder} $\psi$ receives as an input the observation $Y_t$ and previous AIS $\hat{S}_{t-1}$ and generates $\hat{S}_t$. It consists of a linear layer of size $(2,4)$ with ReLU activation, followed by a gated recurrent unit (GRU) with a hidden state size of $4$. 
The hidden state of the GRU constitutes the AIS $\hat{S}_t$ updated recurrently as $\hat{S}_{t} = \psi(\hat{S}_{t-1}, Y_t)$, thus enforcing \eqref{ap2a}. Note that our AIS is independent from the agent's position and action because the target moves independent from the agent.
The decoder is comprised of two separate units, each of which is selected according to the action $U_t$. If $U_t = \xi$, we use the network $\phi^{\text{c}}$ which takes as an input the agent's position $X^{\text{ag}}_t$ and the AIS $\hat{S}_t$ and generates a set of possible terminal costs $\hat{\mathcal{K}}^{\text{c}} := [[C_t|X_t^{\text{ag}}, \hat{S}_t]]$. This network comprises of two linear layers with dimensions $(6,16)$ and $(16,9)$, where the first layer has ReLU activation and the second has sigmoid activation. If $U_t \neq \xi$, we use the network $\phi^{\text{y}}$ which takes the AIS $\hat{S}_t$ as an input and generates the conditional range $\hat{\mathcal{K}}^{\text{y}} :=[[Y_{t+1}|\hat{S}_t]]$.
This network comprises of two linear layers with dimensions $(6,16)$ and $(16,23)$, where the first layer has ReLU activation and the second has sigmoid activation. 

We train the entire model simultaneously using the outputs of the decoder. At each $t \in \mathbb{N}$, the training loss is given by the Hausdorff distance between the one-hot encoded incoming data point, either $C_t$ or $Y_{t+1}$, and the current predicted set.
Since the Hausdorff distance is not differentiable, we adapt the distance-transform-based surrogate loss proposed in \cite{karimi2019reducing}. Note that we cannot observe the true underlying set and thus train the predictions against sampled data points to eventually learn the feasible sets.
We train the network for $3 \times 10^{6}$ instances with a learning rate of $0.0003$. In each instance, we randomly initialize the agent and target's positions from the pink and blue hatched cells in Fig. \ref{fig:1a} and randomize all subsequent noises, disturbances and actions. %Furthermore, we randomly select all actions, observations and disturbances during training. 

Next, we utilize the trained encoder's output AIS and the agent's position as a state input to a deep Q-learning network (DQN) with two layers of $(6,3)$ and $(3,6)$ and a LeakyReLU activation each. We train this AIS-DQN using an exploratory policy for $3 \times 10^{6}$ instances with a learning rate of $0.0005$ using a maximally risk-averse approach from \cite{mihatsch2002risk} with high risk-aversion $\kappa = 0.9$, to learn to minimize the worst-case discounted cost with $\gamma = 0.97$. We compare the worst-case performance of the greedy strategy of the trained AIS-DQN with the worst-case performance of a trained stochastic-DQN, which uses the observation and position as the state and has the same hyperparameters with no risk-aversion.
In Fig. \ref{fig:results}, we present the improvement in worst-case cost achieved by AIS-DQN over stochastic-DQN in $10^4$ simulations each for different initial positions. Note that AIS-DQN outperforms stochastic-DQN for most cases.

%We present a comparative analysis between the performance of the proposed strategy and a baseline strategy that employs the observation $Y_t$ at each time step $t$ instead of the AIS $\hat{s}_t$ as input in a stochastic DQL. For the baseline strategy, the stochastic DQL network was trained with the same parameters as before to determine the worst-case cost between the agent and the target. 
%To investigate the differences between the worst-case costs for the two different DQLs, we considered various initial positions of the agent and the target, as shown in Fig. \ref{fig:illustration}. The results of the comparison are illustrated in Fig. \ref{fig:trajectory_results_IRL}, where the position of the agent is displayed along the x-axis, and the differences are presented as bars. The non-stochastic DQL that employs the AIS outperforms the baseline in most cases, improving the worst-case cost. However, it exhibits a less good performance in scenarios where the agent is situated close to an obstacle, highlighting the need for more exhaustive exploration.

\begin{figure}[ht!]
  \centering
  \vspace{-10pt}
  %\captionsetup{justification=centering}
  \includegraphics[height = 6cm, keepaspectratio]{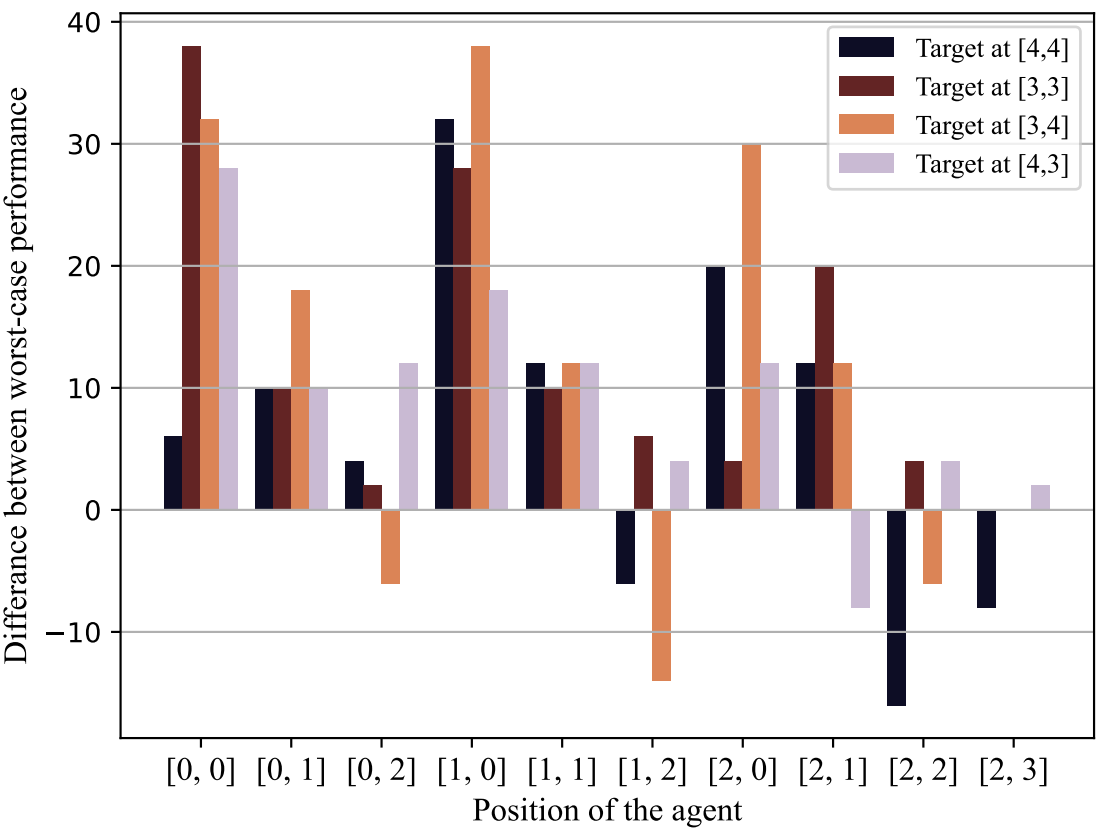} 
  \vspace{-10pt}
  \caption{The improvement in worst-case performance using AIS-DQL over stochastic-DQL.}
  \vspace{-15pt}
  \label{fig:results}
\end{figure}

\section{Conclusions} \label{section:conclusion}

In this paper, we provided a general notion of information states for worst-case decision-making problems over an infinite time horizon without a known state-space model. We showed that these information states yield a time-invariant DP decomposition to compute an optimal control strategy. Then, we specialized this notion to problems with observable costs and extended it to define approximate information states. We proved that approximate information states assist in computing an approximate control strategy with a bounded loss in worst-case performance. 
Finally, we illustrated using a numerical example, how approximate information states can be learned using output data and used to generate control strategies.
Future work should consider using these results in applications to requiring approximately worst-case control and worst-case reinforcement learning.

%and approximate information states to tractably compute control strategies in non-stochastic additive cost problems. We used the theoretical framework of cost distributions to present a general definition for information states that compute an optimal control strategy. We showed that specific information states proposed in previous research efforts emerge as special cases of our definition. Then, we extended this definition to approximate information states which can be used to compute approximate control strategies which admit a bounded worst-case performance loss. Finally, using a numerical simulation, we illustrated the trade-off between computational tractability and performance loss inherent in the application of approximate information states. Future work should consider the use of this theory in non-stochastic reinforcement learning problems.

\bibliographystyle{ieeetr}
\bibliography{References,Latest_IDS}

\newpage

\section*{Appendix A - Proof that the Dynamic Programming Operator is a Contraction Mapping}

In this appendix, we prove that the operator $\mathcal{T}$ defined in \eqref{eq_general_value_operator} is a contraction mapping.

\begin{lemma}
    Consider the operator $\mathcal{T}$ defined in \eqref{eq_general_value_operator}. There exists a constant $\alpha \in [0,1)$ such that for two functions $\Lambda:\mathcal{S} \times (0,1] \to \mathbb{R}$ and $\tilde{\Lambda}:\mathcal{S} \times (0,1] \to \mathbb{R}$:
    \begin{gather} \label{eq_contraction_map}
        ||\mathcal{T}\Lambda - \mathcal{T}\tilde{\Lambda}||_{\infty} \leq \alpha {\cdot} ||\Lambda - \tilde{\Lambda}||_{\infty}.
    \end{gather}    
\end{lemma}

\begin{proof}
    Using the definition of $\mathcal{T}$, we expand the left hand side (LHS) of \eqref{eq_contraction_map} as $||\mathcal{T}\Lambda - \mathcal{T}\tilde{\Lambda}||_{\infty} 
        = \sup_{s \in \mathcal{S}, z \in (0,1]}$ $\big|\inf_{u \in \mathcal{U}} \sup_{s' \in \mathcal{S}, c \in \mathcal{C}}( c + \gamma {\cdot} \Lambda(s', \gamma {\cdot} z) + \frac{\rho(c, s' ~|~ s, u)}{z})$
        $- \inf_{u \in \mathcal{U}} \sup_{s' \in \mathcal{S}, c \in \mathcal{C}}( c + \gamma {\cdot} \tilde{\Lambda}(s', \gamma {\cdot} z) + \frac{\rho(c, s' ~|~ s, u)}{z})\big|$
        $\leq \gamma {\cdot} \sup_{s' \in \mathcal{S}, \gamma {\cdot} z \in (0,\gamma]} |\Lambda(s', \gamma {\cdot} z) - \tilde{\Lambda}(s', \gamma {\cdot} z) | \leq \gamma {\cdot} \sup_{s \in \mathcal{S}, z \in (0,1]} |\Lambda(s, z) - \tilde{\Lambda}(s, z)|$
        $= \gamma {\cdot} ||\Lambda - \tilde{\Lambda}||_{\infty}$,
    where, in the first inequality, we upper bound the difference between supremum values of two functions by the supremum difference between the two functions; and in the second inequality, we use $(0, \gamma] \subset (0,1]$ in the argument of the supremum. This proves that the operator $\mathcal{T}$ is a contraction mapping by setting $\alpha = \gamma$.
\end{proof}

\section*{Appendix B - Proof that Information States Yield an Optimal Control Strategy}

In this appendix, we prove that the information-state based control strategy $\boldsymbol{\pi}^* = (\pi^*, \pi^*, \dots)$ and corresponding memory-based control strategy $\boldsymbol{g}^* = (g_0^*, g_1^*, \dots)$ defined in Subsection \ref{subsection:basic_info_states} are optimal solutions to Problem \ref{problem_1}. Recall from Subsection \ref{subsection:basic_info_states} that the information-based control law is $\pi^*(s, z) = \arg \min_{u \in \mathcal{U}} \sup_{c \in \mathcal{C}, s' \in \mathcal{S}}(c + \gamma {\cdot} \Lambda^{\infty}(s', \gamma {\cdot}z) + \frac{\rho(c, s'|s, u)}{z})$ for all $s \in \mathcal{S}$ and $z \in (0,1]$. Furthermore, recall that the memory-based control law is $g_t^*(m_t) = \pi^*(\sigma_t(m_t), \gamma^t)$ for all $m_t \in \mathcal{M}_t$ and $t \in \mathbb{N}$.
To begin, for any time-invariant control law $\pi: \mathcal{S} \times (0,1] \to \mathcal{U}$, we define a law-dependent operator $\mathcal{T}(\pi):[\mathcal{S} \times (0,1] \to \mathbb{R}] \to [\mathcal{S} \times (0,1] \to \mathbb{R}]$, such that for any uniformly bounded function $\Lambda:\mathcal{S} \times (0,1] \to \mathbb{R}$, we have:
\begin{multline} \label{eq_pi_dependent_operator}
    [\mathcal{T}(\pi) \Lambda](s, z) \\
    := \sup_{c \in \mathcal{C}, s' \in \mathcal{S}} \Bigg(c + \gamma {\cdot} \Lambda (s', \gamma {\cdot} z)
    + \frac{\rho(c, s'|s, \pi(s,z) )}{z} \Bigg). 
\end{multline}
Note that the control action in the RHS of \eqref{eq_pi_dependent_operator} is selected using as $u = \pi(s, z)$. Furthermore, by definition of $\pi^*$ it holds that for all $s \in \mathcal{S}$ and $z \in (0,1]$:
\begin{gather} \label{eq_appendix_b_pi_star}
    [\mathcal{T}(\pi^*) \Lambda^{\infty}](s, z) = [\mathcal{T} \Lambda^{\infty}](s,z) = \Lambda^{\infty}(s,z).
\end{gather}
Then, we can construct a corresponding memory-based control law at each ${{t \in \mathbb{N}}}$ as $g_t(m_t) := \pi(\sigma_t(m_t), \gamma^t)$ and a memory-based control strategy $\boldsymbol{g} = (g_0, g_1, \dots)$. Next, we establish that we can use $[\mathcal{T}(\pi)^n \Lambda^0](\sigma_t(m_t), \gamma^t)$ for any ${{n \in \mathbb{N}}}$ to estimate the strategy-dependent value $V^{\boldsymbol{g}}_t(m_t)$ at each ${{t \in \mathbb{N}}}$.

\begin{lemma} \label{lem_appendix_b}
    For all ${{t \in \mathbb{N}}}$ and all ${{n \in \mathbb{N}}}$, it holds that:
    \begin{multline} \label{eq_lem_appendix_b}
        \frac{\gamma^{n+t}{\cdot} c^{\min}}{1-\gamma} + \gamma^t {\cdot} [\mathcal{T}(\pi)^n \Lambda^0](\sigma_t(m_t), \gamma^t) + \sup_{a_t \in [[A_t|m_t]]}a_t \\
        \leq V^{\boldsymbol{g}}_t(m_t) \\ 
        \leq \sup_{a_t \in [[A_t|m_t]]}a_t + \gamma^t {\cdot} [\mathcal{T}(\pi)^n \Lambda^0](\sigma_t(m_t), \gamma^t)  + \frac{\gamma^{n+t} {\cdot} c^{\max}}{1-\gamma}.
    \end{multline}
\end{lemma}

\begin{proof}
    The proof follows using the same sequence of arguments as the proof for Theorem \ref{thm_g_opt}, but by using \eqref{eq_finite_DP_memory_g} from Lemma \ref{finite_DP_memory} along with $u_t = g_t(m_t):= \pi(\sigma_t(m_t), \gamma^t)$.
\end{proof}

We can set $t=0$ in \eqref{eq_lem_appendix_b} and note that
\begin{multline}
    \frac{\gamma^{n}{\cdot} c^{\min}}{1-\gamma} + [\mathcal{T}(\pi)^n \Lambda^0](\sigma_0(y_0), 1) \; \leq \; V^{\boldsymbol{g}}_t(m_t) \\ 
    \leq [\mathcal{T}(\pi)^n \Lambda^0](\sigma_t(m_t), 1) + \frac{\gamma^{n} {\cdot} c^{\max}}{1-\gamma},
\end{multline}
where recall that $A_0 = 0$ always. Thus, as a direct consequence of Lemma \ref{lem_appendix_b}, it holds that
$[\lim_{n \to \infty} \mathcal{T}(\pi)^n] \Lambda^0(\sigma_0(y_0), 1) = V_0^{\boldsymbol{g}}(y_0)$.
Next, we prove that $\lim_{n \to \infty}\mathcal{T}(\pi^*)^n = \Lambda^{\infty}$. 

\begin{lemma} \label{lem_appendix_b_2}
    For all $s \in \mathcal{S}$ and $z \in (0,1]$, it holds that
    \begin{gather} \label{eq_appendex_b_2}
        [\lim_{n \to \infty}\mathcal{T}(\pi^*)^n](s,z) = \Lambda^{\infty}(s,z).
    \end{gather}
\end{lemma}

\begin{proof}
    We begin by showing that the LHS of \eqref{eq_appendex_b_2} forms an upper bound on the RHS. By definition, note that $[\mathcal{T}(\pi)\Lambda](s, z) \geq [\mathcal{T}\Lambda](s, z)$ for any control law $\pi$ and for all $s \in \mathcal{S}$, $z \in (0,1]$. Taking the limit on both sides with $\pi = \pi^*$, this implies for all $s \in \mathcal{S}$, $z \in (0,1]$ that
    \begin{gather} \label{eq_b_2_1}
        [\lim_{n \to \infty} \mathcal{T}(\pi^*)^n \Lambda^{{{0}}}](s,z) \geq [\lim_{n \to \infty}\mathcal{T}^n \Lambda^{{{0}}}](s,z) = \Lambda^{\infty}(s,z).
    \end{gather} 
    Next, we prove that the LHS of  \eqref{eq_appendex_b_2} also forms a lower bound on the RHS. From \eqref{eq_appendix_b_pi_star}, it holds that $\Lambda^{\infty} = \mathcal{T}(\pi^*) \Lambda^{\infty} = \lim_{n \to \infty} \mathcal{T}(\pi^*)^n \Lambda^{\infty}$. Then, using $\Lambda^\infty \geq \Lambda^0 = 0$, we write for all $s \in \mathcal{S}$ and $z \in (0,1]$ that
    \begin{gather} \label{eq_b_2_2}
        \Lambda^{\infty}(s,z) = \lim_{n \to \infty} \mathcal{T}(\pi^*)^n \Lambda^{\infty} \geq [\lim_{n \to \infty} \mathcal{T}(\pi^*)^n \Lambda^{{{0}}}](s,z).
    \end{gather}
    Using \eqref{eq_b_2_1} and \eqref{eq_b_2_2} simultaneously establishes \eqref{eq_appendex_b_2}.
\end{proof}

Then, as a direct consequence of both Lemmas \ref{lem_appendix_b} and \ref{lem_appendix_b_2}, we can conclude that
\begin{multline}
    V_t^{\boldsymbol{g}^*}(y_0) = [\lim_{n \to \infty} \mathcal{T}(\pi)^n] \Lambda^0(\sigma_0(y_0), 1) \\
    = \Lambda^{\infty}(\sigma_0(y_0), 1) = V_0(y_0),
\end{multline}
where in the last equality, we use \eqref{eq_optimality_of_infor_state}. This proves that the control strategies $\boldsymbol{g}^*$ and $\boldsymbol{\pi}^*$ are optimal solutions to Problem \ref{problem_1}.

\section*{Appendix C - Proof that the Alternate Characterization Defines an Approximate Information State}

In this appendix, we prove that the alternate characterization presented in Subsection \ref{subsection:alternate} using the properties \eqref{ap2a} and \eqref{ap2b} is sufficient to establish \eqref{eq_def_approximate_info} in Definition \ref{def_approximate_info}.

\begin{lemma}
    For all ${{t \in \mathbb{N}}}$, if an uncertain variable $\hat{S}_t = \hat{\sigma}_t(M_t)$ satisfies \eqref{ap2a} - \eqref{ap2b}, it also satisfies \eqref{eq_def_approximate_info}.
\end{lemma}

\begin{proof}
Let $m_t \in \mathcal{M}_t$ be a given realization of $M_t$ and let $\hat{s}_t = \hat{\sigma}_t(s_t)$ satisfy \eqref{ap2a} - \eqref{ap2b}, for all $t$. Let $\mathcal{K}^\text{ob}_{t} :=[[C_t, Y_{t+1}|m_t,u_t]]$ and $\hat{\mathcal{K}}^\text{ob}_{t} :=[[C_t, Y_{t+1}|\hat{\sigma}_t(m_t),u_t]]$. 
Then, using \eqref{ap2a}, we can write the LHS in \eqref{eq_def_approximate_info} as
$\mathcal{H}\big([[C_t, \psi(\hat{\sigma}_t(m_t), u_t, Y_{t+1})|m_t, u_t]], [[C_t, \psi(\hat{\sigma}_t(m_t), u_t, $ $Y_{t+1})|\hat{\sigma}_t(m_t), u_t]] \big) = \max \big\{ \sup_{(c_t, y_{t+1}) \in \mathcal{K}^\text{ob}_t}  \inf_{(\hat{c}_t,\hat{y}_{t+1}) \in \hat{\mathcal{K}}^\text{ob}_{t}}$ $\big(\eta(c_t, \hat{c}_t) + \eta\big(\psi(\hat{\sigma}_t(m_t),u_t,y_{t+1}), \psi(\hat{\sigma}_t(m_t),u_t,\hat{y}_{t+1})\big)\big),$ $ \sup_{(\hat{c}_t,\hat{y}_{t+1}) \in \hat{\mathcal{K}}^\text{ob}_{t}}$ $\inf_{(c_t, y_{t+1}) \in \mathcal{K}^{\text{ob}}_{t}} \big(\eta(c_t, \hat{c}_t) + \eta\big(\psi(\hat{\sigma}_t(m_t),u_t,$ $y_{t+1}), \psi(\hat{\sigma}_t(m_t),u_t,\hat{y}_{t+1})\big)\big)\big\},$ where, in the second equality, we use the definition of the Hausdorff distance from \eqref{H_met_def}. Note that $\psi$ is globally Lipschitz from the alternate characterization of the approximate information state. This implies that $\eta\big(\psi(\hat{\sigma}_t(m_t),u_t,y_{t+1}),$ $\psi(\hat{\sigma}_t(m_t),u_t,\hat{y}_{t+1})\big) \leq L_{\psi} {\cdot} \eta(y_{t+1},$ $\hat{y}_{t+1})$, and thus 
$\mathcal{H}\big([[C_t, \psi(\hat{\sigma}_t(m_t), u_t, Y_{t+1})|m_t, u_t]], [[C_t, \psi(\hat{\sigma}_t(m_t), u_t, $ $Y_{t+1})|\hat{\sigma}_t(m_t), u_t]] \big) \leq L_{\psi} \max\big\{ \sup_{(c_t,y_{t+1}) \in \mathcal{K}^{\text{ob}}_{t}}$ $\inf_{(\hat{c}_t,\hat{y}_{t+1}) \in \hat{\mathcal{K}}^\text{ob}_{t}} \big(\eta(c_t, \hat{c}_t) + \eta(y_{t+1}, \hat{y}_{t+1}\big), 
\sup_{(\hat{c}_t,\hat{y}_{t+1}) \in \hat{\mathcal{K}}^\text{ob}_{t}}$ $\inf_{(c_t,y_{t+1}) \in \mathcal{K}^{\text{ob}}_{t}}$ $\big(\eta(c_t, \hat{c}_t) + \eta(y_{t+1}, \hat{y}_{t+1}\big)\big\}
= L_{\psi} {\cdot} \mathcal{H}(\mathcal{K}^\text{ob}_{t},\hat{\mathcal{K}}^\text{ob}_{t}) \leq L_{\psi} {\cdot} \delta$. Thus, \eqref{eq_def_approximate_info} is satisfied by selecting $\epsilon = L_{\psi} {\cdot} \delta$.
\end{proof}

\end{document}